\documentclass[12pt]{article}

\usepackage{graphics,amsmath,amssymb,amsthm,mathrsfs}

\newcommand{\varep}{\varepsilon}

\newtheorem{thm}{Theorem}[section]
\newtheorem{lemma}[thm]{Lemma}

\newtheorem{definition}[thm]{Definition}
\newtheorem{remark}[thm]{Remark}

\numberwithin{equation}{section}

\begin{document}

\bibliographystyle{amsplain}

\title{Homogenization of Elliptic Boundary Value Problems \\
in Lipschitz Domains}

\author{Carlos E. Kenig\thanks{Supported in part by NSF grant DMS-0456583}
 \and Zhongwei Shen\thanks{Supported in part by NSF grant DMS-0855294}}

\date{Dedicated to the Memory of Bj\"orn Dahlberg}

\maketitle
\begin{abstract}
In this paper we study the $L^p$ boundary value problems for
$\mathcal{L}(u)=0$ in $\mathbb{R}^{d+1}_+$, where $\mathcal{L}=-\text{div}
(A\nabla )$ is a second order elliptic operator with real and symmetric
coefficients.
Assume that $A$ is {\it periodic} in $x_{d+1}$
and satisfies some minimal smoothness
condition in the $x_{d+1}$ variable,
we show that the $L^p$ Neumann and regularity problems are uniquely solvable
for $1<p<2+\delta$. 
We also present a new proof of Dahlberg's theorem on the 
$L^p$ Dirichlet problem for $2-\delta<p< \infty$
(Dahlberg's original unpublished proof is given in the Appendix).
As the periodic and smoothness conditions
 are imposed only on the $x_{d+1}$ variable,
these results extend directly from $\mathbb{R}^{d+1}_+$
to regions above Lipschitz graphs.
Consequently, by localization techniques,
we obtain uniform $L^p$ estimates for the Dirichlet, Neumann and regularity
problems on bounded Lipschitz domains
for a family of second order elliptic
operators arising in the theory of homogenization.
The results on the Neumann and regularity problems
are new even for smooth domains.

\bigskip


\bigskip


\end{abstract}

\section{Introduction}

Let $\mathcal{L}=-\text{div} (A\nabla )$
be a second order
elliptic operator
defined in $\mathbb{R}^{d+1}=\big\{ X=(x,t)\in \mathbb{R}^{d}\times \mathbb{R}\big\}$, $d\ge 2$.
We will always assume that the $(d+1)\times (d+1)$ coefficient matrix
\begin{equation}\label{real-symmetric}
A=A(X)=(a_{i,j}(X)) \text{ is real and symmetric,}
\end{equation}
and satisfies the ellipticity condition,
\begin{equation}\label{ellipticity}
\mu |\xi|^2 \le a_{ij}(X)\xi_i\xi_j \le \frac{1}{\mu} |\xi|^2
\quad \text{for all }X,\xi\in \mathbb{R}^{d+1},
\end{equation}
where $\mu>0$.
In this paper we shall be interested 
in boundary value problems for $\mathcal{L}(u)=0$ in the upper half-space
$\mathbb{R}^{d+1}_+ =\mathbb{R}^{d} \times (0,\infty)$
with $L^p$ boundary data, under the assumption that the coefficients
are periodic in the $t$ variable,
\begin{equation}\label{periodic-in-t}
A(x,t+1) =A(x,t) \quad \text{ for } (x,t)\in \mathbb{R}^{d+1}.
\end{equation}
More precisely, we will study the solvabilities of the
$L^p$ Dirichlet problem $(D)_p$
\begin{equation}\label{Lp-Dirichlet problem}
\left\{
\aligned 
& \mathcal{L}(u)  =0 \quad \text{in } \Omega=\mathbb{R}^{d+1}_+,\\
&  u =f\in L^p(\partial\Omega) \  \text{ n.t. on } \partial\Omega
\text{ and } (u)^*\in L^p(\partial\Omega),
\endaligned
\right.
\end{equation}
the $L^p$ Neumann problem $(N)_p$
\begin{equation}\label{Neumann-problem}
 \left\{
\aligned 
& \mathcal{L}(u)=0 \quad \text{ in } \Omega=\mathbb{R}^{d+1}_+,\\
& \frac{\partial u}{\partial\nu}=f\in L^p(\partial\Omega) \ \text{ on } \partial\Omega
 \text{ and }
N(\nabla u)\in L^p(\partial\Omega),
\endaligned
\right.
\end{equation}
where $\frac{\partial u}{\partial \nu}$ denote the
conormal derivative associated with operator $\mathcal{L}$,
and the $L^p$ regularity problem $(R)_p$
\begin{equation}\label{regularity-problem}
\left\{
\aligned 
& \mathcal{L}(u)=0 \quad \text{in } \Omega=\mathbb{R}^{d+1}_+,\\
& u=f\in \dot{W}^{1,p}(\partial\Omega)\ \text{n.t. on } \partial\Omega
\text{ and } N(\nabla u)\in L^p(\partial\Omega).
\endaligned
\right.
\end{equation}
Here $(u)^*$ denotes the usual
nontangential maximal function of $u$ and $N(\nabla u)$ a generalized 
nontangantial maximal function of $\nabla u$.
By $u=f$ n.t. on $\partial\Omega$ we mean that
$u(X)$ converges to $f(P)$ as $X\to P$ nontangentially for a.e. $P\in \partial\Omega$. 
Under the periodic condition (\ref{periodic-in-t}) as well as some (necessary) local
solvability conditions on $\mathcal{L}$, we will show that
the $L^p$ Dirichlet problem is
uniquely solvable for $2-\delta<p<\infty$, and the $L^p$ Neumann and regularity
problems are uniquely solvable for $1<p<2+\delta$.
Furthermore, the solution to the Dirichlet problem satisfies the
estimate $\|(u)^*\|_p \le C\| u\|_p$, while the solutions to the
$L^p$ Neumann and regularity problems satisfy 
$\| N(\nabla u)\|_p \le C\|\frac{\partial u}{\partial \nu}\|_p$
and $\|N(\nabla u)\|_p\le C\|\nabla_x u\|_p$, respectively.
These results extend the work
on the $L^p$ boundary value problems in $\mathbb{R}^{d+1}_+$
for second order elliptic operators with 
$t$-independent coefficients (or in bounded star-like Lipschitz domains
for operators with radially independent 
coefficients).
As we shall discuss later,
since the local solvability conditions may be deduced from certain minimal smoothness
conditions in the $t$-variable and our periodic condition is only imposed
on the $t$-variable, the $L^p$ global estimates on $(u)^*$ and $N(\nabla u)$
 extend directly
from $\mathbb{R}^{d+1}_+$ to the regions above Lipschitz graphs.
As a consequence, by well known localization techniques,
we obtain uniform estimates for the $L^p$ boundary
value problems in bounded Lipschitz domains
for a family of second order elliptic operators arising in
the theory of homogenization.

We should point out that the result mentioned above for the Dirichlet problem
is in fact due to the late B. Dahlberg \cite{dahlberg-personal} (unpublished, personal communication).
His proof, which is given in the Appendix for the sake of reference,
depends on the ingenious use of Green's functions and harmonic measures.
As in the case of his celebrated theorem on the Dirichlet problem
for $\Delta u=0$ in Lipschitz domains \cite{Dahlberg-1977, Dahlberg-1979},
 Dahlberg's method does not extend
to the Neumann and regularity problems.
Motivated by the work of Jerison and Kenig \cite{Jerison-Kenig-1980,
Jerison-Kenig-1981-Ann, Jerison-Kenig-1981-Bull},
we then seek to establish the global $L^2$ Rellich type estimate,
\begin{equation}\label{Rellich-type-estimate}
\| \frac{\partial u}{\partial \nu}\|_2
\approx \|\nabla_x u\|_2
\end{equation}
for suitable solutions of elliptic operators with $t$-periodic coefficients.
In Section 3 we develop some new integral identities which play a crucial role
in the proof of (\ref{Rellich-type-estimate}) and which may be regarded as the Rellich
identities for operators with $t$-periodic coefficients.
Indeed, in the case of constant or $t$-independent coefficients, Rellich identities
are usually derived by using integration by parts on a form involving $\frac{\partial u}{\partial t}$.
The basic insight here is to replace the $t$ derivative of $u$ by
the difference $Q(u)(x,t)=u(x,t+1)-u(x,t)$.
The $t$-periodicity of $A(x,t)$ is used in the fact that $Q$ commutes with $\mathcal{L}$.
In particular, $Q(u)$  is a solution
whenever $u$ is a solution.
We further remark that these integral identities allow us to control the near-boundary
integral
\begin{equation}\label{boundary-estimate}
\int_0^1 \int_{\mathbb{R}^{d}} |\nabla u(x,t)|^2\, dxdt,
\end{equation}
by the integral of $|\frac{\partial u}{\partial\nu}|^2$ or $|\nabla_x u|^2$
on $\mathbb{R}^d$. The desired estimate (\ref{Rellich-type-estimate}) then follows
by local solvability conditions.

With (\ref{Rellich-type-estimate}) at our disposal, we are able to give another proof
of Dahlberg's theorem on the Dirichlet problem in Section 4 and
more importantly,
solve the $L^2$ regularity problem in Section 5 and the $L^2$ Neumann problem
in Section 8.
Furthermore, using the $L^2$ estimates and
following the approach developed in \cite{Dahlberg-Kenig-1987}
and \cite{KP-1993}, we establish the
solvability of the $L^p$ regularity and Neumann problems for
$1<p<2$ in Sections 6 and 9.
Although the range $2<p<2+\delta$ may be also treated by the real variable arguments
used in \cite{Dahlberg-Kenig-1987, KP-1993}, we choose to use a relatively new real variable approach,
based on the weak reverse H\"older inequalities \cite{shen-2006-ne, shen-2007-boundary}.
In particular it allows us to show that for elliptic operators with coefficients
satisfying (\ref{real-symmetric})-(\ref{ellipticity}),
if the $L^2$ regularity problem $(R)_2$ is solvable 
and $\frac{1}{p}+\frac{1}{q}=1$,
then the Dirichlet problem $(D)_q$ and the regularity problem $(R)_p$ are equivalent.

There exists an extensive literature on boundary value problems
with minimal smoothness assumptions on the coefficients or on the boundary
of the domain in question.
In particular the $L^p$ Dirichlet, regularity and Neumann problems for
$\Delta u=0$ in Lipschitz domains, which are closely related to
the problems we investigate here,
 were well understood thanks to \cite{Dahlberg-1977, Dahlberg-1979, Jerison-Kenig-1980,
Jerison-Kenig-1981-Ann, Jerison-Kenig-1981-Bull, Verchota-1984, Dahlberg-Kenig-1987}.
Deep results on the $L^p$ Dirichlet problem   
for a general second order elliptic equation $\mathcal{L}(u)=0$
may be found in \cite{CFMS-1981, Dahlberg-1986, Fe-1989, FKP-1991}
(see \cite{Kenig-1994} for further references).
The $L^p$ regularity and Neumann problems for $\mathcal{L}(u)=0$
were formulated and studied in \cite{KP-1993, KP-1995}.
If the coefficient matrix $A$ is $t$-independent, the $L^p$ solvability
of the Neumann and regularity problems in 
the upper half-space was essentially established in \cite{KP-1993}
for the sharp range $1<p<2+\delta$, although the paper only treats the case of
radially independent coefficients in the unit ball.
We also mention that the $L^2$ boundary value problems
for some operators with $t$-independent complex coefficients was recently studied
by the method of layer potentials in \cite{Hofmann-preprint}.
Related work on $L^p$ boundary value problems for
operators with $t$-independent, but non-symmetric coefficients
may be found in \cite{KKPT-2000, KR-2009}.
Regarding the conditions on the $t$ variable, it worths 
mentioning that the global estimate $\| (u)^*\|_p \le C\| f\|_p$
may fail for any $p$, even if the matrix $A(x,t)$ satisfies
(\ref{real-symmetric})-(\ref{ellipticity}) and is $C^\infty$ in $\mathbb{R}^{d+1}$.

Since the local solvability condition on $\mathcal{L}$ may be deduced from
certain minimal smoothness assumption in the $t$ variable on the coefficient matrix $A$, as 
we alluded to earlier,
Dahlberg's theorem on the Dirichlet problem and our results on the Neumann and regularity problems
extend readily to the case where 
\begin{equation}\label{Lipschitz-region}
D
=\big\{ (x,t)\in \mathbb{R}^{d+1}: \, t>\psi(x)\big\}
\end{equation}
is the region above a Lipschitz graph, using the bi-Lipschitzian map
$(x,t)\to (x, t-\psi(x))$ from $D$ to $\mathbb{R}^{d+1}_+$.
Indeed, let
\begin{equation}\label{eta-function}
\eta(\rho)=\sup\big\{ |A(x,s_1)-A(x,s_2)|: \
x\in \mathbb{R}^d \text{ and } |s_1-s_2|\le \rho\big\}
\end{equation}
and assume that
\begin{equation}
\label{smoothness-condition-in-t}
\int_0^1 \frac{ \big( \eta (\rho)\big)^2 }{\rho}\, d\rho<\infty.
\end{equation}
The following three theorems are proved in Section 10.
The constant $\delta$ in Theorems \ref{Dirichlet-theorem-Lipschitz},
\ref{Neumann-theorem-Lipschitz} and \ref{regularity-theorem-Lipschitz}
depends only on $d$, $\mu$, $\eta(\rho)$
and $\|\nabla_x\psi\|_\infty$.

\begin{thm}\label{Dirichlet-theorem-Lipschitz}
Let $\mathcal{L}=-\text{div}(A\nabla )$
with coefficient matrix satisfying (\ref{real-symmetric}),
(\ref{ellipticity}), (\ref{periodic-in-t})
 and (\ref{smoothness-condition-in-t}).
Let $D$ be given by (\ref{Lipschitz-region}).
Then there exists $\delta\in (0,1)$ such that 
for any $f\in L^p(\partial D)$ with $2-\delta<p<\infty$,
there exists a unique solution to $\mathcal{L}(u)=0$ in
$D$ with the property that $(u)^*\in L^p(\partial D)$
and $u=f$ n.t. on $\partial D$.
Moreover, the solution $u$ satisfies
the estimate $\|(u)^*\|_p \le C_p \| f\|_p$,
where $C_p$ depends only on $d$, $p$, $\mu$, $\eta(\rho)$ and $\| \nabla_x \psi\|_\infty$.
\end{thm}

\begin{thm}\label{Neumann-theorem-Lipschitz}
Under the same assumptions on $\mathcal{L}$ and $D$ as in Theorem
\ref{Dirichlet-theorem-Lipschitz}, there exists $\delta>0$ such that given any
$g\in L^p(\partial D)$ with $1<p<2+\delta$, there exists a
solution, unique up to constants, to $\mathcal{L}(u)=0$ in $D$
such that $N(\nabla u)\in L^p(\partial D)$ and $\frac{\partial
u}{\partial\nu}=g$ on $\partial D$.  Moreover, the solution
satisfies the estimate $\|N(\nabla u)\|_p \le C_p \| g\|_p$, where $C_p$ depends
only on $d$, $p$, $\mu$, $\eta(\rho)$ and $\| \nabla_x \psi\|_\infty$.
\end{thm}

Let $\nabla_{tan} u$ denote the tangential derivatives of $u$ on $\partial\Omega$.
We say $f\in \dot{W}^{1,p} (\partial\Omega)$ if $f\in L^p_{loc} (\partial\Omega)$
and $\nabla_{tan} f \in L^p(\partial\Omega)$.

\begin{thm}\label{regularity-theorem-Lipschitz}
Under the same assumptions on $\mathcal{L}$ and $D$ as in Theorem \ref{Dirichlet-theorem-Lipschitz},
there exists $\delta>0$ such that given any $f\in \dot{W}^{1,p}(\partial D)$
with $1<p<2+\delta$, there exists a unique solution
to $\mathcal{L}(u)=0$ in $D$ such that
$N(\nabla u)\in L^p(\partial D)$ and
$u=f$ n.t. on $\partial D$.
Moreover, the solution satisfies the estimate
$\|N(\nabla u)\|_p \le C_p \| \nabla_{tan} f\|_p$, where
$C_p$ depends only on $d$, $p$, $\mu$, $\eta(\rho)$ and $\| \nabla_x \psi\|_\infty$.
\end{thm}

As we indicated earlier, Theorem \ref{Dirichlet-theorem-Lipschitz}
as well as Theorem \ref{e-Dirichlet-theorem} below are due to B. Dahlberg \cite{dahlberg-personal}.
However, under the $t$-periodic condition,
Theorems \ref{Neumann-theorem-Lipschitz} 
and \ref{regularity-theorem-Lipschitz} are new even for operators
with smooth coefficients on smooth domains.
We further point out that by well known localization techniques,
Theorems \ref{Dirichlet-theorem-Lipschitz},
\ref{Neumann-theorem-Lipschitz} and
\ref{regularity-theorem-Lipschitz}  allow one to establish uniform estimates
for the $L^p$ boundary value problems in Lipschitz domains
for a family of second order elliptic operators arising in the theory of 
homogenization (see e.g. \cite{bensoussan-1978}).

More precisely, we consider
\begin{equation}\label{L-operator}
\mathcal{L}_\varep =-\text{div} \left( A\left(\frac{X}{\varep}\right)\nabla \right), \quad \varep>0,
\end{equation}
where the matrix $A(X)$ is periodic with respect to the standard lattice,
\begin{equation}\label{periodicity}
A(X+Z)=A(X) \quad \text{ for any } X\in \mathbb{R}^{d+1},\ Z\in \mathbb{Z}^{d+1}.
\end{equation}
In view of (\ref{smoothness-condition-in-t}), we will assume that $A(X)$ is 
continuous and its modulus of continuity satisfies
\begin{equation}\label{smoothness-condition}
\int_0^1
\frac{(\omega(\rho))^2}{\rho}\, d\rho
<\infty,
\end{equation}
where $\omega(\rho)=\sup \big\{ |A(X)-A(Y)|:\ X,Y\in \mathbb{R}^{d+1}
\text{ and }
|X-Y|\le  \rho\big\}$.

The proofs of the following three theorems are given in Section 11.

\begin{thm}\label{e-Dirichlet-theorem}
Assume that the coefficient matrix $A(X)$ is real, symmetric, and satisfies the ellipticity
condition (\ref{ellipticity}), the periodicity condition (\ref{periodicity})
and the smoothness condition (\ref{smoothness-condition}).
Let $\Omega$ be a bounded Lipschitz domain in $\mathbb{R}^{d+1}$. 
Then there exist constants $\delta\in (0,1)$ and $C>0$ independent of $\varepsilon$, such that if
$f\in L^p(\partial\Omega)$
for some $2-\delta<p<\infty$, then
the unique solution to the $L^p$ Dirichlet problem:
$\mathcal{L}_\varepsilon (u_\varepsilon)  =0  \text{ in } \Omega,\
 u_\varepsilon
 =f \text{ n.t. on } \partial\Omega,\ 
(u_\varepsilon)^*  \in L^p(\partial\Omega),
$
satisfies the estimate $\|(u_\varepsilon)^*\|_p \le C\| f\|_p$.
\end{thm}

Let $\frac{\partial u}{\partial \nu_\varepsilon}$ denote the conomal derivative
associated with the operator $\mathcal{L}_\varepsilon$.

\begin{thm}\label{e-Neumann-theorem}
Let $\Omega$ be a bounded Lipschitz domain in $\mathbb{R}^{d+1}$. 
Assume that $A(X)$ satisfies the same assumption as in Theorem \ref{e-Dirichlet-theorem}.
Then there exist constants $\delta\in (0,1)$ and $C>0$ independent of $\varepsilon$, such that
given any $g\in L^p(\partial\Omega)$ with $1<p<2+\delta$ and mean value zero,
the unique solution to the $L^p$ Neumann problem:
$
\mathcal{L}_\varepsilon (u_\varepsilon) =0  \text{ in } \Omega,\
 \frac{\partial u_\varepsilon}{\partial \nu_\varepsilon}
 =g \text{ on } \partial\Omega,\
N(\nabla u_\varepsilon) \in L^p(\partial\Omega),
$
satisfies the estimate $\|N(\nabla u_\varepsilon)\|_p \le C \| g\|_p$.
\end{thm}

\begin{thm}\label{e-regularity-theorem}
Let $\Omega$ be a bounded Lipschitz domain in $\mathbb{R}^{d+1}$. 
Assume that $A(X)$ satisfies the same assumption as in Theorem \ref{e-Dirichlet-theorem}.
Then there exist constants $\delta\in (0,1)$ and $C>0$ independent of 
$\varepsilon>0$, such that given any $f\in W^{1,p}(\partial\Omega)$
with $1<p<2+\delta$,
the unique solution to the $L^p$ regularity problem:
$
\mathcal{L}_\varepsilon (u_\varepsilon)  =0  \text{ in } \Omega,\ 
 u_\varepsilon
 =f \text{ n.t. on } \partial\Omega,\
N(\nabla u_\varepsilon) \in L^p(\partial\Omega)
$
satisfies the estimate 
$\|N(\nabla u_\varepsilon)\|_p \le C\, \big\{
\|\nabla_{tan} f\|_p +|\partial\Omega|^{\frac{1}{1-d}}\| f\|_p\big\}.
$
\end{thm}

We remark that Dahlberg's work on operators with periodic coefficients
was inspired by a series of remarkable papers \cite{AL-1987, AL-1987-ho, AL-1989-II, AL-1989-ho, 
AL-1991} of M. Avellaneda and F. Lin
on elliptic homogenization problems.
In particular, 
 the uniform estimate
$\|(u_\varepsilon)^*\|_p \le C\| f\|_p$ in Theorem \ref{e-Dirichlet-theorem}
for $1<p<\infty$ was established 
on $C^{1, \gamma}$ domains
by Avellaneda and Lin in \cite{AL-1987, AL-1987-ho, AL-1991}
for second order elliptic equations and systems  in divergence form with $C^\alpha$
periodic coefficients.
Related work on the boundedness of Riesz transforms
$\nabla (\mathcal{L}_\varepsilon)^{-1/2}$ on Lipschitz and $C^1$ domains
may be found in \cite{shen-2008-estimates}.
However, to the authors' best knowledge, 
the uniform $L^p$ estimates for periodic operators have not been studied
before for the regularity and Neumann problems; Theorems \ref{e-Neumann-theorem} and
\ref{e-regularity-theorem}
are new even for smooth domains.

Finally we mention that as in the case of operators with constant coefficients,
the Rellich estimates we develop in this paper, can be used to solve the $L^p$ boundary
value problems by the method of layer potentials.
This would enable us to represent the solutions
in Theorems \ref{e-Dirichlet-theorem}, \ref{e-Neumann-theorem} and \ref{e-regularity-theorem}
in terms of layer potentials with density functions uniformly bounded in appropriate spaces. 
Even more importantly,
 the method of layer potentials can be applied to elliptic systems
and would allow us to establish uniform $L^p$ estimates
for elliptic systems with periodic coefficients in nonsmooth domains.
This line of ideas will be fully developed in a forthcoming paper \cite{kenig-shen}.

\medskip

\noindent{\it Acknowledgment.}
The first named author is indebted to the late Bj\"orn Dahlberg for sharing with him the proof
of Theorem \ref{Dahlberg-theorem} on the Dirichlet problem.
Our work was motivated in part by Dahlberg's theorem.
Dahlberg died suddenly on January 30, 1998 and did not publish his proof.
As we mentioned earlier, we include a version of his original proof in the Appendix
for future reference.
We dedicate this paper to the memory of Bj\"orn Dahlberg.

\section{Notations and Preliminaries}

We will use $X,Y,Z$ to denote points in $\mathbb{R}^{d+1}$ and
 $x, y, z$ points in $\mathbb{R}^{d}$. 
Balls in $\mathbb{R}^d$ are denoted by
$B(x,r)=\{ y\in \mathbb{R}^d: |y-x|<r\}$.
By $T(x,r)$ we mean the cylinder
$B(x,r)\times (0,r)$.
We will use $\nabla$ for the gradient in $\mathbb{R}^{d+1}$
and $\nabla_x$ for the gradient
in $\mathbb{R}^d$.

\medskip

\noindent{\bf The classical Dirichlet problem in $\Omega=\mathbb{R}^{d+1}_+$.}

Let $f\in C(\mathbb{R}^d)$ such that $f(x)\to 0$ as $|x|\to \infty$. Then there exists
a unique weak solution to $\mathcal{L}(u)=0$ in $\mathbb{R}^{d+1}_+$ such that
$u\in C(\mathbb{R}^{d+1}\times [0,\infty))$, $u(x,0)=f(x)$ for $x\in \mathbb{R}^d$
and $u$ is bounded in $\mathbb{R}^{d+1}_+$. Moreover, the solution 
satisfies $\sup\{ |u(x,t)|: x\in \mathbb{R}^d, t>0\} \le \| f\|_\infty$.
The uniqueness follows from Louiville's Theorem by a reflection argument.
To establish the existence, one may assume that $f\ge 0$.
Let $\Omega_k=T(0,k)=B(0,k)\times (0,k)$ and $\varphi$ a 
continuous decreasing function on $[0,\infty)$ such that
$0\le \varphi\le 1$, $\varphi(s)=1$ for $0\le s\le (1/2)$ and $\varphi(s)=0$
for $s>(3/4)$.
Let $u_k$ be the solution to the classical Dirichlet problem
for $\mathcal{L}(u)=0$ in $T(0,k)$ with boundary data $f(x)\varphi(|x|/k)$
on $\{ (x,0): |x|<k\}$ and zero otherwise.
Since $u_k\le u_{k+1}$ on $\Omega_k$ and $\| u_k\|_{L^\infty(\Omega_k)}\le \| f\|_\infty$,
it follows that
$u(x)=\lim_{k\to\infty} u_k(x)$ exists and is finite for any $x\in \mathbb{R}^{d+1}_+$
and $\| u\|_{L^\infty(\mathbb{R}^{d+1})} \le \| f\|_\infty$.
Also, we may deduce from the estimate
$$
\sup_{\Omega_\ell } |u_k-u_m| \le C (u_k-u_m)(0, \ell/2), \quad \text{for }k>m>\ell
$$
that as $k\to \infty$, 
$u_k$ converges to $u$ uniformly on $\overline{\Omega_\ell}$ for any $\ell>1$.
This implies that $\mathcal{L}(u)=0$ in $\Omega=\mathbb{R}_+^{d+1}$,
$u$ is continuous in $\overline{\Omega}$ and
$u(x,0)=f(x)$.

By the Riesz representation theorem, there exists a family of regular
Borel measures $\{ \omega^X: X\in \mathbb{R}^{d+1}_+\}$ on $\mathbb{R}^d$, called 
$\mathcal{L}$-harmonic measures,
such that
\begin{equation}\label{Riesz-representation}
u(X)=\int_{\mathbb{R}^d} f(y)\, d\omega^X (y).
\end{equation}
Let $v_k$ be the solution to the classical Dirichlet problem in $\mathbb{R}^{d+1}_+$
with data
$\varphi(|x|/k)$.
It follows from the uniqueness that $v_k(X)\to 1$ as $k\to \infty$ 
for any $X\in \mathbb{R}^{d+1}_+$.
This shows that $\omega^X$ is a probability measure, as in the case of
bounded domains.
Furthermore, let $\omega_k^X$ denote the $\mathcal{L}$-harmonic measure
for the cylinder $B(0,k)\times (0,k)$. It is not hard to see that
$\omega_k^X(B)\to \omega^X (B)$ as $k\to\infty$, for any ball $B\subset \mathbb{R}^d$
and
$X\in \mathbb{R}^{d+1}_+$.

\medskip

\noindent{\bf Green's function on $\Omega=\mathbb{R}^{d+1}_+$.}

Let $G_k(X,Y)$ denote the Green's function for $\mathcal{L}$ on the domain
$\Omega_k=T(0,k)$.
Since $0\le G_k(X,Y)\le G_{k+1}(X,Y)$ for $X,Y\in \Omega_k$, 
the Green's function 
for $\mathcal{L}$ on $\Omega=\mathbb{R}^{d+1}_+$ may be defined by
$$
G(X,Y)=\lim_{k\to\infty} G_k (X,Y)
\quad \text{for } X,Y \in \mathbb{R}_+^{d+1} \text{ and } X\neq Y.
$$
By well known properties of $G_k$, we have
\begin{equation}\label{Green's-function-estimate}
0\le G(X,Y)  \le \frac{C}{|X-Y|^{d-1}}\quad \text{ and }\quad
G(X,Y)  \le \frac{C t^{\alpha}}{|X-Y|^{d-1+\alpha}}
\end{equation}
for $X=(x,t)$, $Y\in \mathbb{R}_+^{d+1}$, where
$C$ and $\alpha$ are positive constants depending only on $d$ and $\mu$.
Since $\omega_k^X(B(y,r))\approx r^{d-1} G_k (X,(y,r))$ for any $X\in \Omega_k
\setminus T(y,2r)$ \cite{CFMS-1981}, by letting $k\to \infty$, 
we obtain
\begin{equation}\label{harmonic-measure-Green's-function}
\omega^X (B(y,r)) \approx r^{d-1} G(X, (y,r))
\quad \text{ for any } X\notin T(y,2r).
\end{equation}

\medskip

\noindent{\bf Two properties of nonnegative weak solutions.}

Let $u$ be a nonnegative weak solution of $\mathcal{L}(u)=0$ in $T(x_0,2r)$
for some $x_0\in \mathbb{R}^d$ and $r>0$.
Suppose that $u\in C(\overline{T(x_0,2r)})$ and
$u(x,0)=0$ on $B(x_0,2r)$. Then we have the boundary Harnack inequality,
\begin{equation}\label{local-bound}
u(x,t)\le C u(x_0, r) \quad \text{ for any } (x,t)\in T(x_0,r).
\end{equation}
We will also need the following comparison principle,
\begin{equation}\label{comparison-principle}
\frac{u(x,t)}{v(x,t)}
\le C\frac{u(x_0,r)}{v(x_0,r)}
\quad \text{ for any } (x,t)\in T(x_0,r),
\end{equation}
where $u$, $v$ are two nonnegative weak solutions in $T(x_0,2r)$
such that $u, v\in C(\overline{T(x_0,2r)})$ and
$u(x,0)=v(x,0)=0$ on $B(x_0,2r)$ (see e.g. \cite{CFMS-1981}).
The constants in (\ref{local-bound}) and (\ref{comparison-principle})
depend only on $d$ and $\mu$.

\medskip

\noindent {\bf Difference operator $Q$.}

Define
\begin{equation}\label{difference operator}
Q(u) (x,t)=u(x,t+1)-u(x,t) \quad \text{for }(x,t)\in \mathbb{R}^{d+1}.
\end{equation}
Clearly, $Q(\frac{\partial u}{\partial x_i})=\frac{\partial}{\partial x_i} Q(u)$
for $i=1, \dots,d+1$.
It is also easy to verify that
\begin{equation}\label{product-rule}
Q(uv)-Q(u) Q(v)=uQ(v)+v Q(u)
\end{equation}
and
\begin{equation}\label{difference-property}
\iint_{T(x,t)} Q(u)(y,s)\, dyds
=\int_t^{t+1}\int_{B(x,t)} u(y,s)\, dyds
-\int_0^1 \int_{B(x,t)} u(y,s)\, dyds.
\end{equation}

\medskip

\noindent{\bf Nontangential maximal functions.}

For a function $u$ defined in $\mathbb{R}^{d+1}_+$,
we let
\begin{equation}\label{nontangential-maximal}
N(u) (x)=\sup\left\{ \left(\frac{1}{t |B(x,t)|}
\int_{\frac{t}{2}}^{\frac{3t}{2}}
\int_{B(x,t)}
|u(y,s)|^2\, dyds\right)^{1/2}: \ 0<t<\infty\right\}.
\end{equation}
This is a variant of the usual nontangential maximal function $(u)^*$ which is defined by
\begin{equation}\label{nontangential-maximal-1}
(u)^*(x)=\sup\big\{ |u(y,t)|:\ 0<t<\infty\text{ and } |y-x|<t\big\}.
\end{equation}
Clearly, $N(u)(x)\le (u)^*(x)$.
If $u$ is a weak solution to $\mathcal{L}(u)=0$ in $\mathbb{R}^{d+1}_+$ or any
function that has the property 
$$
|u(x,t)|\le C\left(\frac{1}{r^{d+1}}\iint_{B((x,t),r)}
|u(y,s)|^2\, dyds\right)^{1/2},
$$
whenever $B((x,t),r)\subset \mathbb{R}^{d+1}_+$, then $\| N(u)\|_p \approx \| (u)^*\|_p$
for any $0<p<\infty$.
We also need to introduce $N_r(u)$ and $(u)_r^*$ which are defined by
restricting the variable $t$ in (\ref{nontangential-maximal})
and (\ref{nontangential-maximal-1}) respectively to $0<t<r/2$.
The definitions of $N(u)$ and $(u)^*$ extend naturally to regions above Lipschitz graphs
and to bounded Lipschitz domains.

We conclude this section with a lemma which allows us to estimate
$\| N(\nabla u)\|_p $ by $\| N_4(\nabla u)\|_p$ and $\|(Q(u))^*\|_p$.

\begin{lemma}\label{maximal-function-lemma}
Let $u\in C^1(\overline{T(0,3R)})$ be a solution to $\mathcal{L}(u)=0$ 
in $T(0,3R)$ for some $R>10$.
Let $f(x)=u(x,0)$.
Then for any $x\in B(0,R)$,
\begin{equation}\label{maximal-function-estimate}
N_R(\nabla u)(x)
\le 
C\,  M\bigg(\big[N_4(\nabla u)+|\nabla_x f|+(Q(u))_{R}^*\big]\chi_{B(0,3R)}\bigg)(x),
\end{equation}
 where $M$ denotes the Hardy-Littlewood
maximal operator on $\mathbb{R}^d$ and $C$ depends only on $d$ and $\mu$.
\end{lemma}

\begin{proof} We begin by fixing  $(x_0,t_0)\in \mathbb{R}^{d+1}$ with $x_0\in B(0,R)$
and $t_0\in (2, R/2)$.
Note that by Cacciopoli's inequality and interior regularity of weak solutions,
\begin{equation}\label{2-13}\aligned
\bigg(\frac{1}{t_0^{d+1}} &\int_{|s-t_0|<\frac{t_0}{2}}
\int_{B(x_0,t_0)} |\nabla u (y,s)|^2\, dyds\bigg)^{1/2}\\
&\le \frac{C}{t_0}
\left(\frac{1}{t_0^{d+1}}\int_{|s-t_0|<\alpha_1 t_0}
\int_{B(x_0,\beta_1 t_0)} |u(y,s)-E|^2\, dyds\right)^{1/2}\\
&\le \frac{C}{t_0^{d+2}}
\int_{|s-t_0|<\alpha_2 t_0}
\int_{B(x_0,\beta_2 t_0)} |u(y,s)-E|\, dyds,
\endaligned
\end{equation}
where $(1/2)<\alpha_1<\alpha_2<(3/4)$, $1<\beta_1<\beta_2<(3/2)$ and $E\in \mathbb{R}$.
To estimate the right-hand side of (\ref{2-13}), we use
$$
|u(y,s)-E|\le |u(y,s)-u(y,0)| +|u(y,0)-E|
$$
and choose $E$ to be the average of $u(y,0)$ over
$B(x_0,\beta_1 t_0)$. By Poincar\'e's inequality, we have
\begin{equation}\label{2-14}
\aligned
 \frac{1}{t_0^{d+2}}
\int_{|s-t_0|<\alpha_2 t_0}
& \int_{B(x_0,\beta_2 t_0)}
|u(y,0)-E|\, dyds\\
& \le \frac{C}{t_0^d}
\int_{B(x_0, \beta_2 t_0)}
|\nabla_x f|\, dx
 \le C M\big( |\nabla_x f| \chi_{B(0,3R)}\big) (x_0).
\endaligned
\end{equation}

Next we observe that for any $s\in (0,2t_0)$,
\begin{equation}\label{2-15}
|u(y,s)-u(y,0)|
\le Ct_0 (Q(u))^*_{R} (y) +\int_0^1 |\nabla u(y,t)|\, dt.
\end{equation}
In view of (\ref{2-13}),
the first term in the right-hand side of (\ref{2-15}) can be handled 
easily by $M\big( (Q(u))^*_{R}\chi_{B(0,3R)}\big)(x_0)$.
 Thus it remains to estimate
$$
\frac{1}{t_0^{d+1}}\int_{B(x_0,\beta_2 t_0)}\int_0^1 |\nabla u(y,t)|\, dtdy.
$$
This may be done by observing that
$$
C\int_{B(x_0,4 t_0)} N_4 (\nabla u)\, dx
\ge 
\int_0^{1}
\int_{B(x_0,2t_0)} |\nabla u(y,s)|\, dyds.
$$
\end{proof}

\begin{remark}\label{remark2.1}\label{maximal-function-estimate-remark}
{\rm It follows from (\ref{maximal-function-estimate}) that
for $1<p<\infty$,
\begin{equation}\label{maximal-function-estimate-1}
\aligned
\int_{B(0,R)} |N_R(\nabla u)|^p\, dx
& \le C\int_{B(0,3R)} |\nabla_x u(x,0)|^p\, dx
+C\int_{B(0,3R)} |N_4 (\nabla u)|^p\, dx\\
&\qquad\qquad
+C\int_{B(0,3R)}
|\big(Q(u)\big)^*_R|^p\, dx
\endaligned
\end{equation}
for any $R>10$. 
Thus, if $u\in C^1(\mathbb{R}^d\times [0,\infty))$ and $\mathcal{L}(u)=0$
in $\mathbb{R}^{d+1}_+$, we may let $R\to \infty$ in (\ref{maximal-function-estimate-1})
to obtain
\begin{equation}\label{maximal-function-estimate-2}
\aligned
\int_{\mathbb{R}^d} |N(\nabla u)|^p\, dx
&\le C\int_{\mathbb{R}^d} |\nabla_x u(x,0)|^p\, dx
+C\int_{\mathbb{R}^d} |N_4(\nabla u)|^p\, dx\\
&
\qquad\qquad 
+C\int_{\mathbb{R}^d} |\big(Q(u)\big)^*|^p\, dx,
\endaligned
\end{equation}
where $C$ depends only on $d$, $p$ and $\mu$.
}
\end{remark}

\section{Integral identities}

In this section we develop some new integral identities
that will play a key role in our approach to
 the $L^2$ boundary value problems in $\mathbb{R}^{d+1}_+$
for elliptic equations
with $t$-periodic coefficients.

\begin{lemma}\label{integral-identity-lemma}
Let $\mathcal{L}=-\text{div}(A\nabla) $
with coefficients satisfying conditions (\ref{real-symmetric})-(\ref{ellipticity}).
Let $\Omega\subset \mathbb{R}^{d+1}$ be a bounded Lipschitz domain
and $u\in W^{1,2}(\Omega)$ a variational solution to the Neumann problem:
$\mathcal{L}(u)=0$ in $\Omega$,
$\frac{\partial u}{\partial\nu}=g\in L^2(\partial\Omega)$.
Assume that $u\in W^{1,2}(\Omega_1)$
where $\Omega_1$ contains $\Omega$ and $\{ (x,t+1): (x,t)\in \Omega\}$.
Then
\begin{equation}\label{integral-identity}
\aligned
&-\int_\Omega a_{ij} 
Q\big(\frac{\partial u}{\partial x_i}\cdot \frac{\partial u}{\partial x_j}\big) \, dX
+\int_\Omega a_{ij} Q\big(\frac{\partial u}{\partial x_i}\big)
\cdot Q\big(\frac{\partial u}{\partial x_j}\big) \, dX\\
&\qquad\qquad =-2 \int_{\partial\Omega} g\,
Q(u) \, d\sigma,
\endaligned
\end{equation}
where indices $i,j$ are summed from $1$ to $d+1$ and we identify $x_{d+1}$ with $t$.
\end{lemma}

\begin{proof}
Since $Q(u)\in W^{1,2}(\Omega)$, it follows from the definition of variational solutions that
$$
\aligned
\int_{\partial\Omega} g\,  Q(u)\, d\sigma
&=\int_\Omega a_{ij} \frac{\partial u}{\partial x_j} \cdot Q\big(\frac{\partial u}{\partial x_i}\big)
\, dX\\
&=\frac12
\int_\Omega
a_{ij} \left\{
\frac{\partial u}{\partial x_j} \cdot Q\big(\frac{\partial u}{\partial x_i}\big)+
\frac{\partial u}{\partial x_i} \cdot Q\big(\frac{\partial u}{\partial x_j}\big)\right\}\, dX,
\endaligned
$$
where we have used the symmetry condition $a_{ij}=a_{ji}$ in the second equation.
In view of (\ref{product-rule}), this gives
$$
\int_{\partial\Omega}
g\,  Q(u) \, d\sigma
=\frac12
\int_\Omega
a_{ij}
\left\{
Q\big(\frac{\partial u}{\partial x_i}\cdot
\frac{\partial u}{\partial x_j}\big)
-Q\big(\frac{\partial u}{\partial x_i}\big)\cdot
Q\big(\frac{\partial u}{\partial x_j}\big)\right\}\, dX,
$$
from which the desired identity follows.
\end{proof}

In the rest of this section we will assume that
$a_{ij}$ satisfy the periodic condition (\ref{periodic-in-t}).
To justify the use of integration by parts, we will also assume that
both $a_{ij}$ and $u$ are sufficiently smooth.
However all constants $C$ in this section depend only on $d$ and $\mu$.

The following lemma is crucial in our argument for the $L^2$
Neumann problem.

\begin{lemma}\label{key-Neumann-lemma}
Let $\mathcal{L}=-\text{div}(A\nabla)$ with $C^\infty$
coefficients
satisfying (\ref{real-symmetric}), (\ref{ellipticity}) and
(\ref{periodic-in-t}).
Let $R>10$ and $u\in C^1(\overline{\Omega})$
 be a solution to $\mathcal{L}(u)=0$ in $\Omega=T(x_0,3R)$
and  $\frac{\partial u}{\partial \nu}=g$
on $B(x_0,3R)\times \{ 0\}$.
Then
\begin{equation}\label{key-Neumann-estimate}
\int_0^1 \int_{B(x_0,R)}
|\nabla u|^2\, dxdt
\le C\, \int_{B(x_0,2R)}
|g|^2\, dx
+\frac{C}{R}
\iint_{T(x_0,3R)\setminus T(x_0,R)}
|\nabla u|^2\, dxdt,
\end{equation}
where $C$ depends only on $d$ and $\mu$.
\end{lemma}

\begin{proof}
We may assume that $x_0=0$.
Let $\Omega_r=T(0,r)=B(0,r)\times (0,r)$ for $r\in (R,2R)$.
By the periodicity assumption (\ref{periodic-in-t}) and (\ref{difference-property}),
we have
$$
\aligned
&-\iint_{\Omega_r} a_{ij} Q\big(\frac{\partial u}{\partial x_i}
\cdot \frac{\partial u}{\partial x_j}\big)\, dxdt
=-\iint_{\Omega_r}
Q\big( a_{ij} \frac{\partial u}{\partial x_i}\cdot \frac{\partial u}{\partial x_j}\big)\, dxdt\\
&=
\int_0^1\int_{B(0,r)}
a_{ij} \frac{\partial u}{\partial x_i}\cdot \frac{\partial u}{\partial x_j}\, dxdt
-\int_{r}^{r+1}\int_{B(0,r)}
a_{ij} \frac{\partial u}{\partial x_i}\cdot \frac{\partial u}{\partial x_j}\, dxdt.
\endaligned
$$
It then follows from (\ref{integral-identity}) that
\begin{equation}\label{4.0}
\aligned
& \int_0^1\int_{B(0,r)}
a_{ij} \frac{\partial u}{\partial x_i}\cdot \frac{\partial u}{\partial x_j}\, dxdt
+\iint_{\Omega_r}
a_{ij} Q\big(\frac{\partial u}{\partial x_i}\big) Q\big(\frac{\partial u}{\partial x_j}\big) dxdt\\
& =
-2\int_{\partial\Omega_r}
\frac{\partial u}{\partial \nu}\cdot Q(u)\, d\sigma
+\int_{r}^{r+1}\int_{B(0,r)}
a_{ij} \frac{\partial u}{\partial x_i}\cdot \frac{\partial u}{\partial x_j}\, dxdt,
\endaligned
\end{equation}
for $r\in (R,2R)$.
By the Cauchy inequality, this leads to
\begin{equation}\label{4.1}
\aligned
\int_0^1 \int_{B(0,r)} |\nabla u|^2\, dxdt
&\le \delta \int_{B(0,r)} |Q(u)(x,0)|^2\, dx
+C_\delta \int_{B(0,r)} |g(x)|^2\, dx
\\
&\quad\quad
+C\int_{B(0,r)} 
|\nabla u(x,r)|^2\, dx
+C\int_{B(0,r)} |Q(u)(x,r)|^2\, dx\\
& \quad\quad
+C\int_r^{r+1} \int_{B(0,r)} |\nabla u|^2\, dx dt
+C\int_0^r \int_{\partial B(0,r)} |\nabla u|^2\, d\sigma dt\\
&\quad\quad
+C\int_0^r \int_{\partial B(0,r)} |Q(u)|^2\, d\sigma dt,
\endaligned
\end{equation}
for any $0<\delta<1$. Using
$$
\aligned
\int_{B(0,r)} |Q(u)(x,t)|^2\,dx
 & \le \int_t^{t+1} \int_{B(0,r)} |\nabla u|^2\, dxds,\\
\int_0^r \int_{\partial B(0,r)}
|Q(u)|^2\, d\sigma dt
&\le \int_0^{r+1} \int_{\partial B(0,r)} |\nabla u|^2\, d\sigma dt,
\endaligned
$$
we obtain from (\ref{4.1}) that
\begin{equation}\label{4.2}
\aligned
\int_0^1 \int_{B(0,R)} |\nabla u|^2\, dxdt
&\le C \int_{B(0,2R)}
|g(x)|^2\, dx
+C\int_{B(0,2R)} |\nabla u(x,r)|^2\, dx\\
&\quad\quad\quad +C\int_{r}^{r+1}\int_{B(0,2R)} |\nabla u|^2\, dxdt\\
&\quad\quad \quad
+C\int_0^{3R}\int_{\partial B(0,r)} |\nabla u|^2\, d\sigma dt,
\endaligned
\end{equation}
for $r\in (R, 2R)$.
The desired estimate now follows by integrating both sides of (\ref{4.2})
with respect to $r$ over the interval $(R,2R)$.
\end{proof}

\begin{remark}\label{key-Neumann-remark}
{\rm Under the same assumption as in Lemma \ref{key-Neumann-lemma},
we also have
\begin{equation}\label{key-Neumann-estimate-1}
\int_0^6 \int_{B(x_0,R)}
|\nabla u|^2\, dxdt
\le C\, \int_{B(x_0,2R)}
|g|^2\, dx
+\frac{C}{R}
\iint_{T(x_0,3R)\setminus T(x_0,R)}
|\nabla u|^2\, dxdt.
\end{equation}
To see (\ref{key-Neumann-estimate-1}), one simply replaces $Q(u)$
with $\widetilde{Q}(u)=u(x,t+6)-u(x,t)$ in the proof.}
\end{remark}

The following lemma is the key in our approach to the $L^2$ regularity problem.

\begin{lemma}\label{key-regularity-lemma}
Under the same assumptions on $\mathcal{L}$
as in Lemma \ref{key-Neumann-lemma}, there exists $C=C(d,\mu )>0$ such that
\begin{equation}\label{key-regularity-estimate}
\int_0^1 \int_{B(x_0,R)}
|\nabla u|^2\, dxdt
\le C\, \int_{B(x_0,2R)}
|\nabla_x f|^2\, dx
+\frac{C}{R}
\iint_{T(x_0,3R)\setminus T(x_0,R)}
|\nabla u|^2\, dxdt,
\end{equation}
where $u\in C^1(\overline{T(x_0,3R)})\cap C^2(T(x_0,3R))$ is a solution
to $\mathcal{L}(u)=0$ in $T(x_0,3R)$ and $u=f$ on $B(x_0,3R)
\times \{ 0\}, \, R>10$.
\end{lemma}

\begin{proof} We may assume that $x_0=0$.
As in the proof of Lemma \ref{key-Neumann-lemma}, we let
$\Omega_r =B(0,r)\times (0,r)$ for $r\in (R,2R)$.
By the periodicity assumption, $Q(u)$ is a solution in $\Omega_r$.
It follows that for $r\in (R,2R)$,
\begin{equation}\label{4.4}
\aligned
&\int_{\partial\Omega_r} \frac{\partial u}{\partial \nu}\cdot Q(u)\, d\sigma\\
&=\int_{\partial\Omega_r}
u\cdot \frac{\partial}{\partial \nu} Q(u)\, d\sigma\\
&=\int_{\partial\Omega_r}
un_i \cdot
\int_0^1 \frac{\partial}{\partial s}
\left\{
a_{ij}(x, t+s) \frac{\partial u}{\partial x_j} (x, t+s)\right\} \, ds\, d\sigma(x,t)\\
&=\int_{\partial\Omega_r}
un_i \cdot
\left\{ \frac{\partial}{\partial t}
\int_0^1 a_{ij}(x, t+s) \frac{\partial u}{\partial x_j} (x, t+s)\, ds
\right\}\, d\sigma(x,t)\\
& =\int_{\partial\Omega_r}
u\big(n_i\frac{\partial}{\partial t}
-n_{d+1} \frac{\partial }{\partial x_i}\big)
\left\{ \int_0^1 a_{ij}(x,t+s) \frac{\partial u}{\partial x_j}(x, t+s)ds
\right\} d\sigma(x,t)\\
& =
-\int_{\partial\Omega_r}
\big(n_i\frac{\partial}{\partial t}
-n_{d+1} \frac{\partial }{\partial x_i}\big)u \cdot
\left\{ \int_0^1 a_{ij}(x,t+s) \frac{\partial u}{\partial x_j}(x, t+s)ds
\right\} d\sigma(x,t).
\endaligned
\end{equation}
Consequently, by the Cauchy inequality, we obtain
\begin{equation}
\aligned
&\big|\int_{\partial\Omega_r} \frac{\partial u}{\partial \nu}\cdot Q(u)\, d\sigma\big|\\
&\le C\int_0^1 \int_{\partial\Omega_r}
|\nabla_{tan} u(x,t)|\, |\nabla u(x,t+s)|\,
d\sigma(x,t) ds\\
&\le \delta \int_0^1 \int_{B(0,r)} |\nabla u (x,t)|^2\, dx dt
+C_\delta \int_{B(0,r)}|\nabla_x u(x,0)|^2\, dx\\
&\qquad
+C\int_{B(0,r)} |\nabla u(x,r)|^2\, dx
+C\int_{r}^{r+1} \int_{B(0,r)} |\nabla u(x,t)|^2\, dxdt\\
&\qquad
+C\int_0^{r+1}\int_{\partial B(0,r)}
|\nabla u (x,t)|^2\, d\sigma(x) dt,
\endaligned
\end{equation}
for any $0<\delta<1$, where $\nabla_{tan}$ denotes the tangential gradient on $\partial\Omega_r$.
This, together with (\ref{4.0}), implies that
\begin{equation}\label{4.5}
\aligned
&\int_0^1\int_{B(0,R)}
|\nabla u|^2\, dxdt\\
&\le C\int_{B(0,2R)} |\nabla_x f|^2\, dx
+C\int_{B(0,2R)} |\nabla u(x,r)|^2\, dx\\
&\qquad
+C\int_0^{3R} \int_{\partial B(0,r)} |\nabla u|^2\, d\sigma dt
+C\int_r^{r+1} \int_{B(0,2R)}
|\nabla u|^2\, dxdt,
\endaligned
\end{equation}
for $r\in (R,2R)$. The estimate (\ref{key-regularity-estimate})
 follows by integrating both sides of
(\ref{4.5}) with respect to $r$ over the interval $(R,2R)$.
\end{proof}

\begin{remark}\label{key-regularity-remark}
{\rm By replacing $Q(u)$ with $\widetilde{Q}(u)(x,t)=u(x,t+6)-u(x,t)$ in the proof of 
Lemma \ref{key-regularity-lemma}, we may obtain
\begin{equation}\label{key-regularity-estimate-1}
\int_0^6 \int_{B(x_0,R)}
|\nabla u|^2\, dxdt
\le C\, \int_{B(x_0,2R)}
|\nabla_x f|^2\, dx
+\frac{C}{R}
\iint_{T(x_0,3R)\setminus T(x_0,R)}
|\nabla u|^2\, dxdt.
\end{equation}}
\end{remark}

\section{The $L^p$ Dirichlet problem: a theorem of Dahlberg}

To solve the $L^p$ Dirichlet problem,
we will impose a local solvability condition on $\mathcal{L}$:
for any $x_0\in \mathbb{R}^{d}$ and $0<r\le 1$,
\begin{equation}\label{local-Dirichlet-condition}
\int_{B(x_0, r)} \limsup_{t\to 0}
\left( \frac{u(x,t)}{t}\right)^2\, dx
\le \frac{C_0}{r^3} \iint_{T(x_0, 2r)} |u(x,t)|^2\, dxdt,
\end{equation}
whenever $u\in W^{1,2}(T(x_0,4r))$
is a nonnegative weak solution of
$\mathcal{L}(u)=0$ in $T(x_0, 4r)$ such that $u\in C(\overline{T(x_0, 4r)})$
and $u(x,0)=0$ on $B(x_0, 4r)$.
Recall that for $Z\in \mathbb{R}^{d+1}_+\setminus T(x,2t)$,
\begin{equation}\label{harmonic-measure-Green's function-1}
\frac{\omega^Z(B(x,t))}{|B(x,t)|}
\approx
\frac{G((x,t), Z)}{t},
\end{equation}
where $\omega^Z$ denotes the $\mathcal{L}$-harmonic measure on $\mathbb{R}^{d+1}_+$ and
$G(X,Z)$ the Green's function on $\mathbb{R}^{d+1}_+$.
By applying the condition (\ref{local-Dirichlet-condition})
to $u(X)=G(X,Z)$ for $Z\in \mathbb{R}^{d+1}_+\setminus T(x_0,4r)$,
we may deduce that $\omega^Z$ is absolutely continuous with respect to the 
Lebesgue measure on $\mathbb{R}^d$.
Furthermore, the kernel function
$$
K (x,Z)
:=\lim_{t\to 0} \frac{\omega^Z(B(x,t))}{|B(x,t)|}
$$
 satisfies the reverse H\"older inequality
\begin{equation}\label{reverse-Holder}
\left(\frac{1}{|B|}\int_B |K(y,Z)|^2\, dy\right)^{1/2}
\le 
\frac{C}{|B|}
\int_B K(y,Z)\, dy,
\end{equation}
for $B=B(x,r)$ with $0<r\le 1$, if $Z\notin T(x, 2r)$.
In fact, by the comparison principle (\ref{comparison-principle}), 
conditions (\ref{local-Dirichlet-condition}) and 
(\ref{reverse-Holder}) are equivalent.

The goal of this section is to show that
with the additional periodic condition (\ref{periodic-in-t}), 
the reverse H\"older inequality (\ref{reverse-Holder})
holds without the restriction $r\le 1$.
As a consequence, the $L^p$ Dirichlet problem for $\mathcal{L}(u)=0$ in 
$\mathbb{R}^{d+1}_+$ is solvable for $2-\delta<p<\infty$, under the assumptions 
(\ref{real-symmetric}), (\ref{ellipticity}), (\ref{periodic-in-t}) and (\ref{local-Dirichlet-condition}).
The following theorem is due to B. Dahlberg \cite{dahlberg-personal}.

\begin{thm}\label{Dahlberg-theorem}
Let $\mathcal{L}=-\text{div}(A\nabla)$ be an elliptic
operator satisfying conditions 
(\ref{real-symmetric}), (\ref{ellipticity}), 
(\ref{periodic-in-t}) and (\ref{local-Dirichlet-condition}).
Then there exists $\delta=\delta(d,\mu, C_0)\in (0,1)$
such that given any $f\in L^p(\mathbb{R}^{d})$
with $2-\delta<p<\infty$,
there exists a unique weak solution to $\mathcal{L}(u)=0$ in $\mathbb{R}^d$
with the property that $u=f$ n.t. on $\mathbb{R}^{d}$ and $(u)^*\in L^p(\mathbb{R}^{d})$.
Moreover the solution satisfies the estimate
$
\| (u)^*\|_p\le C_p\, \| f\|_p,
$
where $C_p$ depends
only on $d$, $p$, $\mu$ and the constant $C_0$ in (\ref{local-Dirichlet-condition}).
\end{thm}

The key step in the proof of Theorem \ref{Dahlberg-theorem}
is the following lemma. 

\begin{lemma}\label{key-Dirichlet-lemma}
Under the same assumption on $\mathcal{L}$
 as in Theorem \ref{Dahlberg-theorem}, 
estimate (\ref{local-Dirichlet-condition}) and hence (\ref{reverse-Holder})
hold for $r>1$.
\end{lemma}
 
\begin{proof}
We will prove this lemma under the additional assumption that the coefficient matrix
$A(X)$ is sufficiently smooth. The assumption allows us to use
estimate (\ref{key-regularity-estimate}).
For the sake of completeness and future reference we shall provide 
a version of Dahlberg's original proof
of Lemma \ref{key-Dirichlet-lemma} in the Appendix.
We point out that our proofs of the theorems stated in the Introduction, given in Sections 10 and 11,
involve an approximation argument and 
do not rely on Dahlberg's proof.

The case $1< r\le 10$ follows easily
by covering $B(x_0,r)$ with balls of radius $1$.
If $r>10$, we also
cover $B(x_0,r)$ with a sequence of balls 
$\{ B(x_k,1)\}$ of radius $1$ such that
$$
B(x_0,r) \subset \bigcup_k B(x_k, 1)\subset B(x_0,r+1)
\quad \text{ and }\quad
\sum_k \chi_{B(x_k,1)}\le C.
$$
Let $I$ denote the left-hand side of (\ref{local-Dirichlet-condition}).
It follows from the assumption (\ref{local-Dirichlet-condition}) (with $r=1$)
and boundary Harnack inequality (\ref{local-bound}) that
\begin{equation}\label{Dirichlet-estimate-2}
\aligned
I &\le C\, \sum_k \int_{T(x_k,2)} |u(X)|^2\, dX \le C\, 
\sum_k \int_{T(x_k,1)} |u(X)|^2 \, dX\\
&\le \sum_k \int_{T(x_k,1)} |\nabla u(X)|^2\, dX
\le C\, \int_0^1 \int_{B(x_0, r+1)} | \nabla u(x,t)|^2\, dxdt.
\endaligned
\end{equation}
In view of Lemma \ref{key-regularity-lemma}, we obtain
$$
\aligned
I&\le \frac{C}{r}\iint_{T(x_0,3.5r)} |\nabla u|^2\, dxdt
\le \frac{C}{r^3}\iint_{T(x_0, 3.6r)} |u|^2\, dxdt\\
&\le \frac{C}{r^3}\iint_{T(x,2r)}
|u|^2\, dxdt,
\endaligned
$$
where we also used Cacciopoli's inequality as well as the Harnack
inequality (\ref{local-bound}).
\end{proof}

For $1<p<\infty$, we say the $L^p$ Dirichlet problem $(D)_p$ for
$\mathcal{L}(u)=0$ in $\mathbb{R}^{d+1}_+$ is solvable
if for any $f\in C_c(\mathbb{R}^d)$, 
the solution to the classical Dirichlet problem  for
$\mathcal{L}(u)=0$ in $\mathbb{R}^{d+1}_+$ with boundary
data $f$ satisfies the estimate
$\|(u)^*\|_p \le C\| f\|_p$.

\begin{lemma}\label{Dirichlet-existence}
 Under the same assumption on $\mathcal{L}$ as in 
Theorem \ref{Dahlberg-theorem}, 
the $L^p$ Dirichlet problem for $\mathcal{L}(u)=0$ in $\mathbb{R}^{d+1}_+$
is solvable for $2-\delta<p<\infty$.
Moreover, we have $\|(u)^*\|_p 
\le C \| f\|_p$ for $2-\delta<p< \infty$, where $\delta=\delta(d,\mu, C_0)\in (0,1)$
and $C=C(d,\mu, C_0, p)>0$.
\end{lemma}

\begin{proof}
By the self-improving property of the reverse H\"older inequality,
there exists $q_0>2$ such that
for $B=B(x_0,r)$ and $Z=(z,z_{d+1})\in \mathbb{R}^{d+1}_+\setminus T(x_0,4r)$,
\begin{equation}\label{reverse-Holder-in-q}
\left\{\frac{1}{|B|}
\int_B |K(x,Z)|^{q_0} \, dx\right\}^{1/q_0}
\le \frac{C}{|B|}
\int_B K(x,Z)\, dx=C\, \frac{\omega^Z(B)}{|B|}
\le \frac{C (z_{d+1})^\alpha}{r^{d+\alpha}},
\end{equation}
where we used (\ref{harmonic-measure-Green's-function})
and (\ref{Green's-function-estimate}).
Since
$$
u(Z)=\int_{\mathbb{R}^d} K(x,Z) f(x)\, dx,
$$
one may deduce from (\ref{reverse-Holder-in-q}) that 
$$
(u)^*(x)
\le C\, \big\{ M (|f|^{p_0}) (x)\big\}^{1/p_0}
$$
for any $x\in \mathbb{R}^d$,
where $p_0 =q_0^\prime<2$
and $M$ denotes the Hardy-Littlewood maximal operator on $\mathbb{R}^d$.
This gives $\| (u)^*\|_p \le C\| f\|_p$ for any $p>p_0=2-\delta$.
\end{proof}

We need the following Cacciopoli-type lemma in the proof of uniqueness.

\begin{lemma}\label{Cacciopoli-type}
Let $u$, $v$ be two weak solutions in $B(x_0,2R)\times (r/2,3r)$ where $R\ge r$.
Suppose that either $u$ or $v$ is nonnegative.
Then
$$
\int_r^{2r} \int_{B(x_0,R)}
|\nabla u| |v|\, dxdt
\le\frac{C}{r}
\int_{\frac{r}{2}}^{3r}
\int_{B(x_0,2R)}
|u||v|\, dxdt,
$$
where $C$ depends only on $d$ and $\mu$.
\end{lemma}

\begin{proof} Assume that $u\ge 0$.
Let $X\in B(x_0,R)\times (r,2r)$.
Using Cauchy inequality and Cacciopoli's inequality, we obtain
$$
\aligned
\int_{B(X, \frac{r}{4})} |\nabla u||v|dY
& \le\frac{C}{r}
\left(\int_{B(X,0.5r)} |u|^2 \, dY\right)^{1/2}
\left(\int_{B(X,0.5r)} |v|^2\, dY\right)^{1/2}\\
& \le \frac{C}{r}\sup_{B(X,0.5r)} u \int_{B(X,0.5r)}
|v|\, dY \le \frac{C}{r}
\int_{B(X,0.5r)} |u||v|\, dY,
\endaligned
$$
where we used the Harnack's inequality in the last step.
The desired estimate now follows by covering $B(x_0,R)\times (r,2r)$
with balls of radius $r/4$.
The proof for the case that $v\ge 0$ is similar.
\end{proof}

The next theorem concerns the uniqueness in the $L^p$ Dirichlet problem in $\mathbb{R}^{d+1}_+$.

\begin{thm}\label{Dirichlet-uniqueness}
Let $\mathcal{L}=-\text{div}(A\nabla)$ with coefficients satisfying
(\ref{real-symmetric})-(\ref{ellipticity}).
Suppose that $d\omega^Z/dx=K(\cdot,Z)\in L^{q}(\mathbb{R}^d)$
for all $Z\in \mathbb{R}^{d+1}_+$ and some $q>1$. Let $u$ be a weak solution to
 $\mathcal{L}(u)=0$ in $\mathbb{R}^{d+1}_+$ such that
$(u)^*\in L^p(\mathbb{R}^d)$, where $p=q^\prime$.
 Assume that
$u(x,t)\to 0$ as $t\to 0^+$ for a.e. $x\in \mathbb{R}^d$.
Then $u\equiv 0$ in $\mathbb{R}^{d+1}_+$.
\end{thm}

\begin{proof}
Fix $Z\in \mathbb{R}^{d+1}_+$.
Let $G(X)=G(X,Z)$ be the Green's function with pole at $Z$.
Choose $\varphi (x) \in C_0^\infty(B(0,\ell/2))$ and
$\psi (t) \in C_0^\infty (1/(2\ell), 2\ell)$ such that
$\varphi=1$ in $B(0,\ell/4)$, $|\nabla \varphi|\le C/\ell$
and
$\psi =1 $ on $(1/\ell, \ell)$, $|\nabla \psi|\le C\ell$ on $(1/(2\ell), 1/\ell)$,
$|\nabla \psi|\le C/\ell$ on $(\ell,2\ell)$.
Then
$$
\aligned
u(Z)
&=\int_{\mathbb{R}^{d+1}_+}
a_{ij} \frac{\partial G}{\partial x_j}\cdot \frac{\partial }{\partial x_i}
\big(u\varphi\psi\big) dY\\
&=\int_{\mathbb{R}^{d+1}_+}
a_{ij} \frac{\partial G}{\partial x_j} u\cdot \frac{\partial }{\partial x_i}
\big(\varphi\psi\big)dY
-\int_{\mathbb{R}^{d+1}_+}
G\cdot a_{ij} \frac{\partial u}{\partial x_i}\cdot \frac{\partial}{\partial x_j}
\big(\varphi\psi\big)\, dY.
\endaligned
$$
This gives
\begin{equation}\label{uniqueness-1}
\aligned
|u(Z)|
& \le C\ell
 \int_{\frac{1}{2\ell}}^{\frac{1}{\ell}} \int_{|y|<\frac{\ell}{2}}
\big( |\nabla G| | u|
+|G||\nabla u|\big) dY\\
&\qquad\qquad
+\frac{C}{\ell}
\int_\ell^{2\ell}\int_{|y|<\frac{\ell}{2}}
\big( |\nabla G| |u|
+|G||\nabla u|\big) dY\\
&\qquad\qquad
+\frac{C}{\ell}\int_{\frac{1}{\ell}}^\ell
\int_{\frac{\ell}{4}<|y|<\frac{\ell}{2}}
\big( |\nabla G| | u|
+|G||\nabla u|\big) dY\\
&\le C\ell^2 \int_\frac{1}{3\ell}^{\frac{2}{\ell}}\int_{|y|< \ell}
|G||u|dY
+\frac{C}{\ell^2}\int_{\frac{\ell}{2}}^{3\ell}\int_{|y|<\ell}
|G||u|dY\\
&\qquad\qquad
+\frac{C}{\ell}\int_0^{2\ell} \int_{\frac{\ell}{8}< |y|< \ell}
|G||u|\, \frac{dydt}{t}\\
&=I_1 +I_2+I_3,
\endaligned
\end{equation}
where we used Lemma \ref{Cacciopoli-type} for the second inequality.

To estimate $I_2$, we note that
 $|G(Y)|\le C (Z)/\ell^{d-1+\alpha}$ for $Y\in B(0,\ell)\times (\ell/2, 3\ell)$, where
$\alpha>0$.
This leads to
\begin{equation}\label{I-2}
I_2\le \frac{C}{\ell^{d+\alpha}}\int_{B(0,3\ell)} (u)^*\,dy
\le C \ell^{-\alpha -\frac{d}{p}} \|(u)^*\|_p \to 0
\quad \text{as } \ell\to\infty.
\end{equation}
Next we observe that
\begin{equation}\label{I-1}
I_1\le C\int_{|y|<\ell}
M\big(K(\cdot,Z)\big) 
\mathcal{M}_{2/\ell} (u) dy,
\end{equation}
where $M$ denotes the Hardy-Littlewood maximal operator on $\mathbb{R}^{d}$ and
\begin{equation}\label{local-maximal-function}
\mathcal{M}_r (u)(x) =\sup \big\{ u(x,t): \ 0<t<r\big\}.
\end{equation}
Since $K(\cdot,Z)\in L^{p^\prime}(\mathbb{R}^d)$, we have
$I_1\le C\, \| \mathcal{M}_{2/\ell} (u)\|_p$.
By the assumption,  $\mathcal{M}_{2/\ell} (u)(x) \to 0$ for a.e. $x\in \mathbb{R}^d$
as $\ell\to\infty$.
Since $\mathcal{M}_{2/\ell}(u)\le (u)^*\in L^p(\mathbb{R}^d)$, we obtain
$\| \mathcal{M}_{2/\ell}(u)\|_p   \to 0$ as $\ell \to \infty$.

Finally, note that as $\ell\to\infty$,
$$
I_3 \le C\int_{\frac{\ell}{4}<|y|<\ell}
M\big(K (\cdot,Z)  \big)
(u)^* dy
\le C
\left\{\int_{|y|>\frac{\ell}{4}} |(u)^*|^p\, dy\right\}^{1/p}
\to 0.
$$
Thus we have proved that $I_1+I_2+I_3\to 0$ as $\ell \to\infty$
Hence $u(Z)=0$ for any $Z\in \mathbb{R}^{d+1}_+$.
\end{proof}

Theorem \ref{Dahlberg-theorem} is a direct
consequence of Lemma \ref{Dirichlet-existence} and the 
following theorem.

\begin{thm}\label{Dirichlet-solvability}
Let $\mathcal{L}=-\text{div}(A\nabla)$ 
with coefficients satisfying (\ref{real-symmetric})-(\ref{ellipticity}).
Suppose that the $L^p$ Dirichlet problem for $\mathcal{L}(u)=0$ in $\mathbb{R}^{d+1}_+$
is solvable for some $1<p<\infty$.
Then for any $f\in L^p(\mathbb{R}^d)$, there exists a unique
weak solution in $\mathbb{R}^{d+1}_+$ such that
$(u)^*\in L^p(\mathbb{R}^d)$ and $u=f$ n.t. on $\mathbb{R}^d$.
Moreover the solution satisfies the estimate
$\|(u)^*\|_p \le C\| f\|_p$.
\end{thm}

\begin{proof}
Let $f\in C_c(\mathbb{R}^d)$ and $u(X)=\int_{\mathbb{R}^d} f d\omega^X$  
be the solution to the classical Dirichlet problem with data $f$.
Since
$$
|u(x,t)|\le\frac{C}{t^d}
\int_{B(x,t)} |(u)^*|dy
\le {C}{t^{-\frac{d}{p}}}\|(u)^*\|_p
\le {C}{t^{-\frac{d}{p}}}\|f\|_p,
$$
it follows that $\omega^X$ is absolutely continuous with respect to $dx$ on $\mathbb{R}^d$
and $K(\cdot,X)=d\omega^X/dx\in L^{p^\prime}(\mathbb{R}^d)$.
In view of Theorem \ref{Dirichlet-uniqueness},
this gives the uniqueness in Theorem \ref{Dirichlet-solvability}.

To establish the existence, we let $f\in L^p(\mathbb{R}^d)$
and choose $f_k \in C_c (\mathbb{R}^d)$ so that
$f_k\to f$ in $L^p(\mathbb{R}^d)$. Let $u_k$ be the solution 
to the classical Dirichlet problem in $\mathbb{R}^{d+1}_+$
with data $f_k$.
It follows from Lemma \ref{Dirichlet-existence} that
$\| (u_k)^*\|_p \le C\| f_k\|_p$
and $\| (u_\ell - u_k)^*\|_p \le C\| f_\ell -f_k\|_p$.
This implies that $u_k$ converges to a weak solution $u$, uniformly on  any compact subset
of $\mathbb{R}^{d+1}_+$.
Standard limiting arguments 
show that $\|(u)^*\|_p \le C\| f\|_p$ and $u=f$ n.t. on $\mathbb{R}^d$.
\end{proof}

\section{$L^2$ estimates for the regularity problem}

\begin{definition}
{\rm
Let $1<p<\infty$.
We say that the $L^p$ regularity problem $(R)_p$ for $\mathcal{L}(u)=0$ in
$\mathbb{R}^{d+1}_+$ 
is solvable, if for any $f\in C^1_0(\mathbb{R}^d)$,
the unique solution to the classical Dirichlet problem
with boundary data $f$  satisfies the estimate
$\|N(\nabla u)\|_p \le C\| \nabla_x f\|_p$.
}
\end{definition}

In this section we study the $L^2$ regularity problem in $\mathbb{R}^{d+1}_+$
under the condition that the coefficient matrix
is periodic in the $t$ direction.
As in the case of the Dirichlet problem, we need a local solvability condition:
for any $x_0\in \mathbb{R}^{d}$ and $0<r\le 1$,
\begin{equation}\label{local-regularity-condition}
\int_{B(x_0,r)} |N_{\frac{r}{2}} (\nabla u)|^2\, dx
\le C_1\, \left\{ 
\int_{B(x_0,2r)} |\nabla_x f|^2\, dx
+\frac{1}{r}
\iint_{T(x_0,2r)}|\nabla u|^2\, dxdt\right\},
\end{equation}
whenever $u\in W^{1,2}(T(x_0,4r))$ is a weak solution to $\mathcal{L}(u)=0$
in $T(x_0,4r)$ such that
$u\in C(\overline{T(x_0,4r)})$ and $u=f\in C^1(B(x_0,4r))$ on $B(x_0,4r)\times \{ 0\}$.
We will also assume that the coefficient matrix $A(X)$ is
sufficiently smooth. However, constants $C$ in all estimates
will depend only on $d$, $\mu$ and $C_1$ in (\ref{local-regularity-condition}).

We begin by observing that since
\begin{equation}\label{generalized-maximal-operator-estimate}
|u(x,t)-u(x,0)|\le Ct N_{2t}(\nabla u)(x)
\end{equation}
(see \cite{KP-1993}, pp.461-462), condition (\ref{local-regularity-condition})
implies the local condition (\ref{local-Dirichlet-condition}) for
the Dirichlet problem.
It also follows from (\ref{local-regularity-condition}) that
\begin{equation}\label{local-regularity-condition-2}
\int_{B(x_0,1)} |N_4 (\nabla u)|^2\, dx
\le C\, \left\{ 
\int_{B(x_0,2)} |\nabla_x f|^2\, dx
+
\iint_{T(x_0,6)}|\nabla u|^2\, dxdt\right\}.
\end{equation}
Furthermore, if 
$\mathcal{L}(u)=0$ in $\mathbb{R}_+^{d+1}$, $u\in C(\mathbb{R}^d\times [0,\infty))$ and
$u=f\in C_0^1(\mathbb{R}^{d})$ on $\mathbb{R}^d$,
we may integrate both sides of (\ref{local-regularity-condition-2}) with respect
to $x_0$ over $ \mathbb{R}^d$ to obtain
\begin{equation}
\label{local-regularity-estimate-3}
\int_{\mathbb{R}^d} |N_4(\nabla u)|^2\, dx
\le C\left\{
\int_{\mathbb{R}^d} |\nabla_x f|^2\, dx
+\int_0^6 \int_{\mathbb{R}^{d}} |\nabla u|^2\, dxdt\right\}.
\end{equation}

\begin{thm}\label{regularity-theorem}
Let $\mathcal{L}=-\text{div}(A\nabla )$
with smooth coefficients satisfying 
conditions (\ref{real-symmetric}), (\ref{ellipticity}),
(\ref{periodic-in-t}).
Also assume that $\mathcal{L}$ satisfies the condition (\ref{local-regularity-condition}).
Then the $L^2$ regularity problem for $\mathcal{L}(u)=0$ in $\mathbb{R}^{d+1}_+$
is solvable.
Moreover the solution $u$ satisfies the estimate
$\|N(\nabla u)\|_2\le C\| \nabla_x f\|_2$, where $C$ depends only on
$d$, $\mu$ and the constant $C_1$ in (\ref{local-regularity-condition}).
\end{thm}

In view of (\ref{key-regularity-estimate-1}), we start with a decay estimate.

\begin{lemma}\label{lemma 5.1}
Under the same assumption on $\mathcal{L}$ as in Theorem \ref{regularity-theorem},
we have
\begin{equation}\label{decay-estimate}
R
\iint_{T(0,10R)\setminus T(0,R)}
|\nabla u|^2\, dxdt \to 0 \quad \text{as } R\to \infty,
\end{equation}
where $u$ is a classical solution to $\mathcal{L}(u)=0$
in $\mathbb{R}^{d+1}_+$ such that
$u=f\in C_c({\mathbb{R}^d})$ on $\mathbb{R}^d$.
\end{lemma}

\begin{proof}
By Cacciopoli's inequality,
\begin{equation}\label{decay-estimate-1}
\aligned
R\iint_{T(0,10R)\setminus T(0,R)}
|\nabla u|^2\, dxdt
&\le \frac{C}{R}\iint_{T(0,11R)\setminus T(0,R/2)}
|u|^2\, dxdt\\
& \le C\int_{|x|\ge \frac{R}{100}} 
|(u)^*|^2\, dx
\endaligned
\end{equation}
for $R$ sufficiently large.
Since $(u)^*\in L^2(\mathbb{R}^d)$, the right-hand side of (\ref{decay-estimate-1})
goes to zero as $R\to\infty$.
\end{proof}

\medskip

\noindent{\bf Proof of Theorem \ref{regularity-theorem}.}
Let $f\in C_0^1(\mathbb{R}^d)$ and $u$ be the solution to the classical
Dirichlet problem with data $f$.
It follows from Remark \ref{maximal-function-estimate-remark} and 
(\ref{local-regularity-estimate-3}) that
\begin{equation}\label{existence-5.1}
\aligned
\int_{\mathbb{R}^d} |N(\nabla u)|^2\, dx
& \le C\left\{ \int_{\mathbb{R}^d} |\nabla_x f|^2\, dx
+\int_{\mathbb{R}^d}
|N_4(\nabla u)|^2\, dx
+\int_{\mathbb{R}^d}
|\big(Q(u)\big)^*|^2\, dx\right\}\\
&\le C\left\{ \int_{\mathbb{R}^d} |\nabla_x f|^2\, dx
+\int_0^6 \int_{\mathbb{R}^d}
|\nabla u|^2\, dxdt
+\int_{\mathbb{R}^d}
|\big(Q(u)\big)^*|^2\, dx\right\}.
\endaligned
\end{equation}

To estimate the integral of $|\nabla u|^2$ over $\mathbb{R}^d\times (0,6)$, we let
$R\to \infty$ in (\ref{key-regularity-estimate-1}). In view of (\ref{decay-estimate}),
we obtain
\begin{equation}\label{existence-5.2}
\int_0^6 \int_{\mathbb{R}^d}
|\nabla u|^2\, dxdt \le C\int_{\mathbb{R}^d} |\nabla_x f|^2\, dx.
\end{equation}

Finally, to handle the term with $Q(u)$ in (\ref{existence-5.1}), we observe that
$u$ is also the unique solution to the $L^2$ Dirichlet problem for
$\mathcal{L}(u)=0$ in $\mathbb{R}^{d+1}_+$ with boundary data $f$. This implies that
$(Q(u))^*\in L^2(\mathbb{R}^d)$ and by Theorem \ref{Dahlberg-theorem},
\begin{equation}\label{existence-5.3}
\aligned
\int_{\mathbb{R}^d}
|\big( Q(u)\big)^*|^2\, dx
&\le C\int_{\mathbb{R}^d} |Q(u)|^2\, dx
\le C\int_0^1 \int_{\mathbb{R}^d} |\nabla u|^2\, dxdt\\
&\le C\int_{\mathbb{R}^d} 
|\nabla_x f|^2\, dx,
\endaligned
\end{equation}
where we used (\ref{existence-5.2}) in the last step.
The desired estimate $\|N(\nabla u)\|_2 \le C\| \nabla_x f\|_2$
now follows from (\ref{existence-5.1}), (\ref{existence-5.2}) and
(\ref{existence-5.3}).
\qed

\medskip

We omit the proof of the next lemma and refer the reader to
\cite{KP-1993} for an analogous result (Theorem 3.1, pp.461-462)
in the case of the unit ball.

\begin{lemma}\label{nontangential-maximal-consequence}
Let $\mathcal{L}=-\text{div}(A\nabla )$ be an elliptic operator
satisfying (\ref{real-symmetric})-(\ref{ellipticity}).
Suppose that $\mathcal{L}(u)=0$ in $\mathbb{R}^{d+1}_+$ and
$N_r(\nabla u)\in L^p(\mathbb{R}^d)$ for some $1<p<\infty$ and $r>0$.
Then $u$ converges n.t. to $f$ on $\mathbb{R}^d$ and
$\nabla_x f\in L^p(\mathbb{R}^d)$. Moreover, let
$$
V_i(x,t) =\frac{1}{t|B(x,t)|}
\int_{\frac{t}{2}}^{\frac{3t}{2}}
\int_{B(x,t)}
\frac{\partial u}{\partial y_i}\, dy ds
$$
for $i=1, \dots, d$, then
$V_i (x,t)$ converges weakly to $\frac{\partial f}{\partial x_i}$
in $L^p(\mathbb{R}^d)$ as $t\to 0$.
In particular, $\|\nabla_x f\|_p \le \|N(\nabla u)\|_p$.
\end{lemma}

The next theorem concerns the uniqueness of solutions to the $L^p$
regularity problem.

\begin{thm}\label{regularity-uniqueness}
Let $\mathcal{L}=-\text{div}(A\nabla)$ 
with coefficients satisfying (\ref{real-symmetric})-(\ref{ellipticity}).
Let $u$ be a weak solution to $\mathcal{L}(u)=0$ in $\mathbb{R}^{d+1}_+$
such that $N(\nabla u)\in L^p(\mathbb{R}^d)$ for some $1\le p< \frac{d}{1-\alpha}$, where
$\alpha>0$ is given by (\ref{Green's-function-estimate}).
Then  $u\equiv 0$ in $\mathbb{R}^{d+1}_+$,
 if $u(x,t)\to 0$ as $t\to 0^+$ for a.e. $x\in \mathbb{R}^d$.
\end{thm}

\begin{proof}
The proof begins in the same way as that of Theorem \ref{Dirichlet-uniqueness}.
As such, this leads to
$|u(Z)|\le I_1 +I_2 +I_3$, where $I_1$, $I_2$ and $I_3$ are defined in
(\ref{uniqueness-1}). To estimate $I_1$,
we use the estimate
(\ref{Green's-function-estimate})
and the observation that $|u(x,t)|\le Ct N(\nabla u)(x)$ for a.e. $x\in \mathbb{R}^d$.
This yields that
$$
I_1 =C\ell^2 \int^{\frac{2}{\ell}}_{\frac{1}{2\ell}}
\int_{|y|<\ell}
|G||u|\, dY
\le \frac{C}{\ell^\alpha}\int_{\mathbb{R}^d}
\frac{N(\nabla u) (y)\, dy}{|(y,0)-Z|^{d-1+\alpha}}
\le C\ell^{-\alpha} \|N(\nabla u)\|_p \to 0
$$
as $\ell\to \infty$, where we used the assumption $p<\frac{d}{1-\alpha}$.
Similarly,
$$
I_2=\frac{C}{\ell^2}
\int_{\frac{\ell}{2}}^{3\ell}\int_{|y|<\ell}
|G||u|\, dydt
\le \frac{C}{\ell^{d-1+\alpha}}
\int_{|y|<\ell} N(\nabla u)\, dy
\le C\ell^{1-\alpha-\frac{d}{p}} \| N(\nabla u)\|_p
$$
which also tends to zero as $\ell\to \infty$, since $p<\frac{d}{1-\alpha}$.
Finally, note that
$$
I_3=\frac{C}{\ell}
\int_0^{2\ell}\int_{\frac{\ell}{4}<|y|<2\ell}
|G||u|\, \frac{dydt}{t}
\le \frac{C}{\ell^{d-1+\alpha}} \int_{|y|<2\ell} |N(\nabla u)|\, dy
\to 0,
$$
as $\ell\to \infty$.
Thus we have proved that $I_1+I_2+I_3 \to 0$ as $\ell\to \infty$.
Hence $u(Z)=0$ for any $Z\in \mathbb{R}^{d+1}_+$.
\end{proof}

Recall that $f\in \dot{W}^{1,p}(\mathbb{R}^d)$ if $f\in W^{1,p}_{loc}(\mathbb{R}^d)$
and $\nabla_x f\in L^p(\mathbb{R}^d)$.
The following theorem provides the existence of solutions with boundary data in
$\dot{W}^{1,p}(\mathbb{R}^d)$.

\begin{thm}\label{regularity-existence}
Let $\mathcal{L}=-\text{div}(A\nabla )$ with coefficients satisfying
(\ref{real-symmetric})-(\ref{ellipticity}).
Suppose that the $L^p$ regularity problem for
$\mathcal{L}(u)=0$ in $\mathbb{R}^{d+1}_+$ is solvable.
Then for any $f\in \dot{W}^{1,p}(\mathbb{R}^d)$, there exists
a solution $u$ to $\mathcal{L}(u)=0$ in $\mathbb{R}^{d+1}_+$
such that $\| N(\nabla u)\|_p \le C\| f\|_p$
 and $u=f$ n.t. on $\mathbb{R}^d$.
\end{thm}

\begin{proof}
Let $f\in \dot{W}^{1,p}(\mathbb{R}^d)$ and choose
$f_k\in C_0^1(\mathbb{R}^d)$ such that $\nabla_x f_k \to \nabla_x f$ in $L^p(\mathbb{R}^d)$.
Let $v_k$ be the solution to the classical Dirichlet problem in $\mathbb{R}^{d+1}_+$
with boundary data $f_k$.
By assumption, we have
$\| N(\nabla v_k)\|_p \le C\| \nabla_x f_k\|_p$
and
$\| N(\nabla v_k -\nabla v_\ell )\|_p \le C\| \nabla_x (f_k-f_\ell)\|_p$.
It follows that for any $m>1$, $\nabla (v_k -v_\ell)\to 0$
in $L^2(\Omega_m)$ as $k, \ell\to\infty$, where
$$
\Omega_m =\left\{ (x,t)\in \mathbb{R}^{d+1}_+: |x|<3m \text{ and }
\frac{1}{3m}< t< 3m\right\}.
$$
Thus there exists $u_m\in W^{1,2}(\Omega_m)$ such that
$v_k -\alpha(k,m) \to u_m$ in $W^{1,2}(\Omega_m)$ as $k\to \infty$, where $\alpha(k,m)$
is the average of $v_k$ over $\Omega_m$. Clearly,
$\mathcal{L}(u_m)=0$ in $\Omega_m$.
By a limiting argument, we also have
\begin{equation}\label{regularity-theorem-1}
\| \widetilde{N}_m (\nabla u_m)\|_{L^p(B(0, m))}
\le C\, \| \nabla_x f\|_p,
\end{equation}
where the nontangential maximal function
$\widetilde{N}_m (\nabla u)$ is defined in a manner similar to $N_m$,
but the variable $t$ in (\ref{nontangential-maximal})
is restricted to $(1/m)<t<m$.

Next, since $u_{m+1}-u_m$ is
constant in $\Omega_m$,
one may define a function $u\in W^{1,p}_{loc}(\mathbb{R}^{d+1}_+)$
such that $u-u_m$ is constant in $\Omega_m$ for any $m\ge 1$.
It follows that $u$ is a weak solution in $\mathbb{R}^{d+1}_+$. 
Moreover, in view of (\ref{regularity-theorem-1}), we have
\begin{equation}\label{regularity-theorem-2.5}
\| \widetilde{N}_m (\nabla u)\|_{L^p(B(0, m))}
\le C\, \| \nabla_x f\|_p.
\end{equation}
Letting $m\to \infty$ in (\ref{regularity-theorem-2.5}) gives
$\| N(\nabla u)\|_p \le C\| \nabla_x f\|_p$.
Similar argument also gives
\begin{equation}\label{regularity-theorem-2}
\|N(\nabla u-\nabla v_k)\|_p \le C\|\nabla_x (f-f_k)\|_p.
\end{equation}

Finally we note that by Lemma \ref{nontangential-maximal-consequence},
$u$ converges n.t. to $h\in \dot{W}^{1,p}(\mathbb{R}^d)$ and
the average of $\frac{\partial u}{\partial x_i}$
over $B(x,t)\times (t/2,3t/2)$ converges weakly to $\frac{\partial h}{\partial x_i}$
in $L^p(\mathbb{R}^{d})$ as $t\to 0$.
This, together with  the 
weak convergence of the average of $\frac{\partial v_k}{\partial x_i}$
to $\frac{\partial f_k}{\partial x_i}$ and estimate (\ref{regularity-theorem-2}),
implies that $\frac{\partial h}{\partial x_i}
=\frac{\partial f}{\partial x_i}$ on $\mathbb{R}^d$ for
$i=1,\dots,d$. It follows that $f-h$ is constant.
Thus, by subtracting a constant, we obtain
$u=f$ n.t. on $\mathbb{R}^d$.
\end{proof}

We conclude this section with two more theorems on the consequences of the solvability
of the regularity problem.

\begin{thm}\label{regularity-imply-Dirichlet}
Let $\mathcal{L}=-\text{div}(A\nabla )$
with coefficients satisfying (\ref{real-symmetric})-(\ref{ellipticity}).
Let $1<p,q<\infty$ and $\frac{1}{p}+\frac{1}{q}=1$.
Then 
for $\mathcal{L}(u)=0$ in $\mathbb{R}^{d+1}_+$,
the solvability of the $L^p$ regularity problem implies
the solvability of the $L^q$ Dirichlet problem.
\end{thm}

\begin{thm}\label{regularity-localization}
Let $\mathcal{L}=-\text{div}(A\nabla )$ 
with coefficients satisfying (\ref{real-symmetric})-(\ref{ellipticity}).
Suppose that the $L^2$ regularity problem for $\mathcal{L} (u)=0$ in
$\mathbb{R}^{d+1}_+$ is solvable.
Let $u\in W^{1,2}(T(x_0,8r))$ be a weak solution in $T(x_0, 8r)$ such that
$N_r (\nabla u)\in L^2(B(x_0,4r))$ and $u=f$ on $B(x_0,4r)$.
Then
\begin{equation}\label{regularity-localization-estimate}
\int_{B(x_0,r)}
|N_{\frac{r}{2}}(\nabla u)|^2\, dx
\le C\int_{B(x_0,2r)} |\nabla_x f|^2\, dx
+\frac{C}{r}
\int_{T(x_0,2r)} |\nabla u|^2\, dX.
\end{equation}
\end{thm}

We omit the proof of both theorems and refer 
the reader to Theorems 5.4 and 5.19 in \cite{KP-1993}, 
 where the analogous
results were proved in the case of the unit ball.
We point out that Theorem 5.19 in \cite{KP-1993}
was stated for solutions in the whole domain.
However an inspection of its proof, which extends readily to the
case of $\mathbb{R}^{d+1}_+$,  shows that the conclusion 
holds for local solutions.
Theorem \ref{regularity-localization} shows that
the local solvability condition (\ref{local-regularity-condition})
is necessary.

\section{$L^p$ estimates for the regularity problem}

In this section we study the $L^p$ regularity problem in $\mathbb{R}^{d+1}_+$
for $1<p< 2+\delta$.

\begin{thm}\label{regularity-theorem-p}
Let $\mathcal{L}=-\text{div}(A\nabla u )$
with coefficients satisfying (\ref{real-symmetric})-(\ref{ellipticity}).
Suppose that $(R)_2$ for $\mathcal{L}(u)=0$
in $\mathbb{R}^{d+1}_+$ is solvable.
Then $(R)_p$ for $\mathcal{L}(u)=0$
in $\mathbb{R}^{d+1}_+$ is solvable for
$1<p<2+\delta$.
\end{thm}

In the case of the unit ball, Kenig and Pipher \cite{KP-1993} proved that
the solvability of $(R)_q$ implies the solvability of
$(R)_p$ for $1<p<q+\delta$.
Although it is possible to extend their results to the case of the upper half-space and
thus obtain the $L^p$ solvability for $1<p<2+\delta$ from the $L^2$ solvability,
we will provide a more direct proof.

For the range $1<p<2$,
we follow the approach used
in \cite{Dahlberg-Kenig-1987} (also see \cite{KP-1993}) by proving
$\| N(\nabla u)\|_1\le C$ for solutions of the $L^2$ regularity problem
with $H^1_{at}$ data.
Our approach to the range $2<p<2+\delta$,
which is based on 
a real variable argument
developed by Shen \cite{shen-2006-ne, shen-2007-boundary},
 is different from that used in \cite{Dahlberg-Kenig-1987, KP-1993}.
The rather general approach reduces the problem to certain weak
reverse H\"older estimates. In particular, our argument
 shows that if $(R)_2$
for $\mathcal{L}(u)=0$ in $\mathbb{R}^{d+1}_+$ is solvable,
then the solvability of $(D)_q$
for some $q<2$ implies
the solvability of $(R)_p$, 
where $\frac{1}{p}+\frac{1}{q}=1$.

We start with a comparison principle.

\begin{lemma}\label{new-comparison-lemma}
Let $u, v\in C(\overline{T(x_0,3r)})$ be two weak solutions in $T(x_0,3r)$
such that $u(x,0)=v(x,0)=0$ on $B(x_0,3r)$. Suppose that
$v$ is nonnegative in $T(x_0,3r)$. Then
$$
|u(x,t)|\le C\cdot \frac{v(x,t)}{v(x_0,r)}
\left(\frac{1}{r^{d+1}}
\iint_{T(x_0,3r)} |u(y,s)|^2\, dyds\right)^{1/2}
$$
for any $(x,t)\in T(x_0,r)$, where $C$ depends only on $d$ and $\mu$.
\end{lemma}

\begin{proof} See Lemma 2.5 in \cite{Shen-2007-relationship}.
\end{proof}

The following is a localization result similar to Theorem \ref{regularity-localization}.
Note that here we assume the solvability of $(D)_2$ instead of the solvability of $(R)_2$.

\begin{thm}\label{regularity-localization-theorem}
Let $\mathcal{L}$ be an elliptic operator with coefficients
satisfying (\ref{real-symmetric})-(\ref{ellipticity}).
Assume that the $L^2$ Dirichlet problem for $\mathcal{L}(u)=0$
in $\mathbb{R}^{d+1}_+$ is solvable.
Then
\begin{equation}\label{regularity-localization-estimate-1}
\int_{B(x_0,r)}
|N_r (\nabla u)|^2\, dx
\le \frac{C}{r^3}
\iint_{T(x_0,4r)} |u(x,t)|^2\, dxdt,
\end{equation}
where $u\in C(\overline{T(x_0,6r)})$ is a weak solution of $\mathcal{L}(u)=0$
in $T(x_0,6r)$ and $u(x,0)=0$ on $B(x_0,6r)$.
\end{thm}

\begin{proof}
Let $v(x,t)$ be the Green's function for $\mathcal{L}$ on $\mathbb{R}^{d+1}_+$
 with pole at $X_0=(x_0,10r)$.
Fix $x\in B(x_0,r_0)$ and $0<t<\frac{r}{2}$.
By Cacciopoli's inequality,
$$
\aligned
&\frac{1}{t|B(x,t)|}\int_{\frac{t}{2}}^{\frac{3t}{2}}
\int_{B(x,t)}|\nabla u|^2\, dyds
 \le \frac{C}{t^3 |B(x,t)|}
\int_{\frac{t}{4}}^{2t}\int_{B(x,2t)}\left|
u(y,s)\right|^2\, dyds\\
&\le \frac{C}{t^3 |B(x_0,t)|}\int_{\frac{t}{4}}^{2t}
\int_{B(x,2t)}
\left|\frac{v(y,s)}{v(x_0,r)}\right|^2 dyds
\cdot\frac{1}{r^{d+1}}
\iint_{T(x_0,4r)} |u(y,s)|^2\, dyds\\
&\le Cr^{2d-2}
\left(\frac{v(x,t)}{t}\right)^2
\cdot\frac{1}{r^{d+1}}
\iint_{T(x_0,4r)} |u(y,s)|^2\, dyds,\\ 
\endaligned
$$
where we have used Lemma \ref{new-comparison-lemma} and $v(x_0,r)\approx r^{1-d}$.
Let $\omega$ denote the $\mathcal{L}$-harmonic measure for $\mathbb{R}^{d+1}_+$, 
evaluated at $(x_0,10r)$.
Since
$$
\frac{v(x,t)}{t}\approx
\frac{\omega(B(x,t))}{|B(x,t)|}
\le C M\big( \chi_{B(x_0,2r)} \frac{d\omega}{dx}\big) (x)
$$
where $M$ denotes the Hardy-Littlewood maximal operator on $\mathbb{R}^d$,
we obtain
\begin{equation}\label{pointwise-estimate-for-regularity}
N_r(\nabla u)(x)
\le Cr^{d-1}
M \big(\chi_{B(x_0,2r)} \frac{d\omega}{dx}\big)(x)
\cdot
\left(\frac{1}{r^{d+1}}
\iint_{T(x_0,4r)} |u(y,s)|^2\, dyds\right)^{1/2}.
\end{equation}
Since the $L^2$ Dirichlet problem for $\mathcal{L}(u)=0$ in $\mathbb{R}^{d+1}_+$
is solvable, the reverse H\"older 
inequality (\ref{reverse-Holder}) for $K=\frac{d\omega}{dx}$ 
holds on $B(x_0,2r)$.
The desired estimate now follows by the $L^2$ boundedness of $M$ and 
 (\ref{reverse-Holder}).
\end{proof}

\begin{remark}\label{weak-reverse-Holder-inequality-remark}
{\rm
Suppose that $(D)_p$ for $
\mathcal{L}(u)=0$ in $\mathbb{R}^{d+1}_+$
is solvable for some $p<2$.
Then the reverse H\"older inequality
(\ref{reverse-Holder-in-q}) holds for $q=p^\prime$.
It follows from (\ref{pointwise-estimate-for-regularity}) and (\ref{reverse-Holder-in-q}) that 
\begin{equation}\label{weak-reverse-1}
\aligned
\left(\frac{1}{r^d}
\int_{B(x_0,r)}
|N_r (\nabla u)|^q\, dx \right)^{1/q}
&\le \frac{C}{r}
\left(\frac{1}{r^{d+1}}
\iint_{T(x_0,4r)} |u(y,s)|^2\, dyds\right)^{1/2}\\
&\le C
\left(\frac{1}{r^{d+1}}
\iint_{T(x_0,4r)} |\nabla u(y,s)|^2\, dyds\right)^{1/2}\\
&\le C
\left(\frac{1}{r^d}
\int_{B(x_0,6r)}
|N (\nabla u)|^2\, dx \right)^{1/2}.
\endaligned
\end{equation}
A simple geometric observation shows that for $x\in B(x_0,r)$,
\begin{equation}\label{geometric-1}
N(\nabla u)(x)
\le N_r (\nabla u)(x)
+\frac{C}{r^d}
\int_{B(x_0,6r)}
|N(\nabla u)|\, dx.
\end{equation}
In view of (\ref{weak-reverse-1}), this gives
\begin{equation}\label{weak-reverse-Holder}
\left(\frac{1}{r^d}
\int_{B(x_0,r)}
|N(\nabla u)|^q\, dx\right)^{1/q}
\le C
\left(\frac{1}{r^d}
\int_{B(x_0,6r)}
|N(\nabla u)|^2\, dx\right)^{1/2}.
\end{equation}
We will need this weak reverse H\"older estimate to treat the case $2<p<2+\delta$
for the regularity problem. 
}
\end{remark}

\begin{thm}\label{regularity-theorem-less-than 2}
Under the same conditions on $\mathcal{L}$ as in Theorem \ref{regularity-theorem-p},
the $L^p$ regularity problem in $\mathbb{R}^{d+1}_+$
is solvable for $1<p<2$.
\end{thm}

\begin{proof}
We follow the approach in \cite{Dahlberg-Kenig-1987}
and prove that $\| N(\nabla u)\|_1 \le C$ for solutions of
the $L^2$ regularity problem with $H^1_{at}$ data.
To this end, let $f$ be a Lipschitz function such that
supp$(f)\subset B(x_0,r)$ and $\|\nabla_x f\|_\infty \le Cr^{-d}$
for some $r>0$ and $x_0\in \mathbb{R}^d$.
Let $u$ be the solution of the $L^2$ regularity problem
with data $f$.
By the $L^2$ estimate $\|N(\nabla u)\|_2 \le C\|\nabla_x f\|_2$,
 one has $\|N(\nabla u)\|_{L^1(B(x_0,Cr))}\le C$.
Thus it suffices to show that if $R\ge r$ and $B(y_0,10R)\cap B(x_0,r)=\emptyset$, 
\begin{equation}\label{regularity-decay-estimate}
\int_{B(y_0,R)} |N(\nabla u)|\, dx \le C\left(\frac{r}{R}\right)^\alpha
\end{equation}
for some $\alpha>0$. This follows from Theorem \ref{regularity-localization-theorem} and
the pointwise estimate
$$
\aligned
|u(X)| &\le \| f\|_\infty \, \omega^X (B(x_0,r))
\le Cr^{1-d} \omega^X (B(x_0,r))\\
& \le C G (X,X_0)\le \frac{Cr^\alpha}{|X-X_0|^{d-1+\alpha}},
\endaligned
$$
where $X_0=(x_0,r)$.
We omit the details.
\end{proof}

If $\Omega$ is a bounded Lipschitz domain, it was proved in \cite{Shen-2007-relationship} that
the solvability of $(R)_p$ for $\mathcal{L}(u)=0$
in $\Omega$ is equivalent
to the solvability of $(D)_{p^\prime}$ for
$\mathcal{L}(u)=0$ in $\Omega$, provided
that $(R)_{p_0}$
for $\mathcal{L}(u)=0$ in $\Omega$ is solvable
for some $p_0>1$.
We extend this result to the case $\Omega=\mathbb{R}^{d+1}_+$.

\begin{thm}\label{regularity-large-than-2}
Let $\mathcal{L}=-\text{div}(A\nabla )$
with coefficients satisfying (\ref{real-symmetric})-(\ref{ellipticity}).
Suppose that $(R)_2$ for
$\mathcal{L}(u)=0$ in $\mathbb{R}^{d+1}_+$
is solvable.
Let $p>2$ and $q=p^\prime$.
Then $(R)_p$ for $\mathcal{L}(u)=0$
in $\mathbb{R}^{d+1}_+$ is solvable if and only if
$(D)_q$ for
$\mathcal{L}(u)=0$ in $\mathbb{R}^{d+1}_+$ is solvable.
\end{thm}

The proof of Theorem \ref{regularity-large-than-2} relies on a real variable
argument which may be formulated as follows.

\begin{thm}\label{real-variable-theorem}
Let $F\in L^2(\mathbb{R}^d)$ and $g\in L^p(\mathbb{R}^d)$ for some $2<p<q$.
Also assume that $g$ has compact support.
Suppose that for each ball $B$ in $\mathbb{R}^d$, there exist
two functions $F_B$ and $R_B$ such that $|F|\le |F_B|+|R_B|$
on $2B$, and
$$
\aligned
\frac{1}{|2B|}\int_{2B} |F_B|^2\, dx
 & \le E_1 \sup_{B^\prime\supset B} \frac{1}{|B^\prime|}
\int_{B^\prime} 
|g|^2\, dx,\\
\left\{\frac{1}{|2B|}
\int_{2B} |R_B|^q\, dx \right\}^{1/q}
&\le E_2 \left\{
\left(\frac{1}{|\beta B|}\int_{\beta B} |F|^2\, dx\right)^{1/2}
+\sup_{B^\prime\supset B}
\left(\frac{1}{|B^\prime|}
\int_{B^\prime}
|g|^2\, dx \right)^{1/2}\right\},
\endaligned
$$
where $E_1, E_2>0$ and $\beta>2$. Then $F\in L^p(\mathbb{R}^d)$ and
$\| F\|_p\le C\| g\|_p$,
where $C$ depends only on $d$, $p$, $q$, $E_1$, $E_2$ and $\beta$.
\end{thm}

\begin{proof} This is a simple consequence
of Theorem 3.2 in \cite{shen-2007-boundary}.
Indeed, it follows from the theorem that
\begin{equation}\label{real-variable-proof-1}
\left\{ \int_{B(0,r)}| F|^p\, dx\right\}^{1/p}
\le C|B(0,r)|^{\frac{1}{p}-\frac12} \left\{ \int_{B(0,Cr)}
|F|^2\, dx \right\}^{1/2}
+C\left\{ \int_{B(0,Cr)} |g|^p\, dx\right\}^{1/p}.
\end{equation}
The estimate $\| F\|_p \le C\| g\|_p$
follows by letting $r\to \infty$ in (\ref{real-variable-proof-1}).
\end{proof}

\medskip

\noindent{\bf Proof of Theorem \ref{regularity-large-than-2}.}
In view of Theorem \ref{regularity-imply-Dirichlet},
we only need to show that the solvability of $(D)_q$ implies the solvability of
$(R)_p$. More precisely,
 it suffices to show that $\| N(\nabla u)\|_p 
\le C\| \nabla_x f\|_p$, if $u$ is the solution of the classical Dirichlet
problem with data $f\in C_0^1(\mathbb{R}^d)$.
This will be done by applying Theorem \ref{real-variable-theorem} to
the functions $F=N(\nabla u)$ and $g=|\nabla_x f|$.

To this end, we fix a ball $B=B(x_0,r)$ in $\mathbb{R}^d$.
Choose $\varphi\in C_0^1 (B(x_0,8r))$ such that $\varphi=1$ in $B(x_0,7r)$
and $|\nabla\varphi|\le Cr^{-1}$. 
Let $\lambda$ be the average of $f$ over $B(x_0,8r)$.
Write $u-\lambda=v+w$, where $v$ is the solution of
the $L^2$ regularity problem with boundary data $(f-\lambda)\varphi$.

We now let $F_B=N(\nabla v)$ and $R_B=N(\nabla w)$.
Clearly, $F=N(\nabla u)\le F_B  + R_B$.
By the $L^2$ regularity estimate,
$$
\aligned
\frac{1}{|2B|}\int_{2B} |F_B|^2\, dx
&\le \frac{1}{|B|}\int_{\mathbb{R}^d} |N(\nabla u)|^2\, dx
\le \frac{C}{|B|}\int_{\mathbb{R}^d}
|\nabla_x \big((f-\lambda)\varphi\big)|^2\, dx\\
&\le \frac{C}{|B|}
\int_{B(x_0,8r)} |\nabla_x f|^2\, dx
\le C\sup_{B^\prime\supset B}
\frac{1}{|B^\prime}
\int_{B^\prime} |g|^2\, dx,
\endaligned
$$
where we have used the Poincar\'e inequality.
This gives the estimate needed for $F_B$.
To verify the condition on $R_B=N(\nabla w)$,
we note that $w$ is the solution of the $L^2$ regularity
problem with data $(f-\lambda)(1-\varphi)$.
Consequently, $w=0$ on $B(x_0,7r)$ and we may use the weak reverse H\"older
estimate (\ref{weak-reverse-Holder}). This leads to
$$
\aligned
&\left\{ \frac{1}{|2B|}
\int_{2B} |R_B|^q\, dx\right\}^{1/q}
=\left\{ \frac{1}{|2B|}
\int_{2B} |N(\nabla w)|^q\, dx\right\}^{1/q}\\
&\le C\left\{ \frac{1}{|B|}\int_{B(x_0,12r)}
|N(\nabla w)|^2\, dx \right\}^{1/2}\\
&\le
C\left\{ \frac{1}{|B|}\int_{B(x_0,12r)}
|N(\nabla u)|^2\, dx \right\}^{1/2}
+C\left\{ \frac{1}{|B|}\int_{B(x_0,12r)}
|N(\nabla v)|^2\, dx \right\}^{1/2}\\
&\le 
C\left\{ \frac{1}{|B(x_0,12r)|}\int_{B(x_0,12r)}
|F|^2\, dx \right\}^{1/2}
+C\sup_{B^\prime\supset B}
\left\{ \frac{1}{|B^\prime|}
\int_{B^\prime}
|g|^2\, dx\right\}^{1/2},
\endaligned
$$
which gives the estimate needed for $R_B$.
Therefore, by Theorem \ref{real-variable-theorem}, we obtain
$\| N(\nabla u)\|_p\ \le C\| \nabla_x f\|_p$ for
$2<p<q$. Finally, since the weak reverse H\"older estimate (\ref{weak-reverse-Holder})
is self-improving, the argument above in fact gives
$\|N(\nabla u)\|_p \le C\|\nabla_x f\|_p$ for $2<p<q+\delta$.  
This finishes the proof.
\qed

\section {A Neumann function on $\mathbb{R}^{d+1}_+$}

Throughout this section we will assume that
$\mathcal{L}=-\text{div}(A\nabla )$ with coefficients  satisfying
(\ref{real-symmetric})-(\ref{ellipticity}).

\begin{definition}
{\rm Let $g\in L^1_{loc} (\mathbb{R}^d)$. We call $u\in W^{1,2}_{loc}(\mathbb{R}_+^{d+1})$
a weak solution to the Neumann problem,
\begin{equation}\label{Neumann-problem-1}
 \mathcal{L}(u)=0 \quad \text{in }\Omega=\mathbb{R}^{d+1}_+ \quad \text{ and } \quad
  \frac{\partial u}{\partial \nu}= g \quad \text{on } \partial\Omega,
\end{equation}
if $|\nabla u|\in L^1(T(0,R))$ for any $R>1$ and 
\begin{equation}\label{variational-formulation}
\int_{\mathbb{R}^{d+1}_+}
a_{ij}\frac{\partial u}{\partial x_j}\frac{\partial \varphi}{\partial x_i}\, dX
=\int_{\mathbb{R}^d} g\varphi \, dx\quad
\end{equation}
for any $\varphi\in C_0^1(\mathbb{R}^{d+1})$.
}
\end{definition}

To construct weak solutions to (\ref{Neumann-problem-1})
 we introduce a Neumann function.
Let $E=(e_{ij})$ be the $(d+1)\times (d+1)$ diagonal matrix with
$e_{11}=\cdots=e_{dd}=1$ and $e_{(d+)(d+1)}=-1$, and
\begin{equation}\label{new-matrix}
\widetilde{A}=\widetilde{A} (x,t)
=\left\{
\aligned
A(x,t) \quad \text{ if } t>0,\\
EA(x,-t)E \quad \text { if } t<0,
\endaligned
\right.
\end{equation}
where $A(x,t)$ is the coefficient matrix for $\mathcal{L}$.
Let $\widetilde{\Gamma}=\widetilde{\Gamma}(X,Y)$ denote
the fundamental solution for the operator $\widetilde{L}
=-\text{div}(\widetilde{A}\nabla)$ on $\mathbb{R}^{d+1}$
with pole at $X$. Using the fact that $\widetilde{L}(v)=0$
if $v(x,t)=u(x,-t)$ and $\widetilde{L}(u)=0$, one may show that
\begin{equation}\label{Green's function symmetry}
\widetilde{\Gamma}(X^*, Y^*)
=\widetilde{\Gamma}(X,Y) \quad \text{ for any }
X, Y\in \mathbb{R}^{d+1},
\end{equation}
where $X^*=(x,-t)$ for $X=(x,t)$. We now define
\begin{equation}\label{Neumann-function-definition}
N(X,Y)= \widetilde{\Gamma}(X,Y)+\widetilde{\Gamma}(X, Y^*) \quad
\text{ for }X, Y\in \mathbb{R}^{d+1}_+.
\end{equation}
Clearly, $N(X,Y)=N(Y,X)$. Since $|X-Y|\le |X-Y^*|$ for
$X, Y\in \mathbb{R}^{d+1}_+$, we also have
\begin{equation}\label{decay of Neumann function}
\aligned
|N(X,Y)| &\le \frac{C}{|X-Y|^{d-1}},\\
|N(X,Y)-N(Z,Y)| &\le \frac{C|X-Z|^{\alpha}}{|X-Y|^{d-1+\alpha}},
\endaligned
\end{equation}
for any $X,Y,Z\in \mathbb{R}^{d+1}_+$,
where $C>0,\, \alpha >0$ depend only on $d$ and $\mu$.

Let $N(X,y)=N(X,(y,0))$ for $y\in \mathbb{R}^d$.
The next lemma shows that $N(X,Y)$ is a Neumann function for
the operator $\mathcal{L}$ on $\mathbb{R}^{d+1}_+$.

\begin{lemma}\label{Neumann-function-lemma}
(a) For $1\le p\le 2$,
let $g\in L^p(\mathbb{R}^d)$ and
\begin{equation}\label{Neumann-representation}
u(X)=
\left\{
\begin{aligned}
 \int_{\mathbb{R}^d} N(X,y) g(y)\, dy \quad \quad & \text{ for } d\ge 3 \text{ or } d=2,\ 1\le p<2,\\
  \int_{\mathbb{R}^d} \big( N(X,y)-N(X_0,y)\big) g(y)\, dy \quad \quad & \text{ for }d= 2 \text { and } p=2,\\
\end{aligned}
\right.
\end{equation}
where $X_0=(0,1)\in \mathbb{R}^{d+1}_+$.
Then $u$ is a weak solution to the Neumann problem on $\mathbb{R}^{d+1}_+$
with data $g$.

(b) Let $F\in C_c(\mathbb{R}^{d+1})$ and
\begin{equation}\label{inhomogeneous-Neumann}
w(X)=\int_{\mathbb{R}^{d+1}_+} N(X,Y) F(Y)\, dY.
\end{equation}
Then
\begin{equation}\label{inhomogeneous-Neumann-1}
\int_{\mathbb{R}^{d+1}_+}
a_{ij} \frac{\partial w}{\partial x_j}\frac{\partial \varphi}{\partial x_i}\, dX
=\int_{\mathbb{R}^{d+1}_+} F\varphi\, dX
\end{equation}
for any $\varphi \in C_0^1(\mathbb{R}^{d+1})$.

(c) If $g\in L^2 (\mathbb{R}^d)$ and has compact support, 
then the solution $u$ defined by
(\ref{Neumann-representation}) belongs to
$W^{1,2}(T(0,R))$ for any $R\ge 1$.
\end{lemma}

\begin{proof} Part (a) follows from
\begin{equation}\label{Neumann-function-lemma-1}
\varphi (Y)
=2\int_{\mathbb{R}^{d+1}_+} a_{ij} \frac{\partial }{\partial x_j}
\big\{ \widetilde{\Gamma}(X, Y)\big\}
\frac{\partial \varphi}{\partial x_i}\, dX
\end{equation}
for $Y=(y,0)$ and $\varphi\in C_0^1(\mathbb{R}^{d+1})$.
To see this, one uses
$$
\widetilde{\varphi} (Y)=\int_{\mathbb{R}^{d+1}} \widetilde{a}_{ij}
\frac{\partial }{\partial x_j}
\big\{ \widetilde{\Gamma}(X,Y)\big\} \frac{\partial \widetilde{\varphi}}{\partial x_i}\, dX,
$$
where $\widetilde{\varphi}$ is the even reflection of $\varphi$ from $\mathbb{R}^{d+1}_+$
to $\mathbb{R}^{d+1}$, i.e.,
$\widetilde{\varphi}(X)=\varphi (X) $ if $x_{d+1}\ge 0$ and
$\widetilde{\varphi}(X)=\varphi (X^*)$ if $x_{d+1}<0$.

To see (b), we write
$$
w(X)=\int_{\mathbb{R}^{d+1}} \widetilde{\Gamma}(X,Y)
\widetilde{F}(Y)\, dY
$$
where $\widetilde{F}$ is the even reflection of $F$ and deduce (\ref{inhomogeneous-Neumann-1})
from
$$
\int_{\mathbb{R}^{d+1}}
\widetilde{a}_{ij}
\frac{\partial w}{\partial x_j} \frac{\partial \widetilde{\varphi}}
{\partial x_i}\, dX
=\int_{\mathbb{R}^{d+1}} \widetilde{F} \widetilde{\varphi}\, dX.
$$

Part (c) follows by duality from part (b). 
Let $h\in C_0^1(\mathbb{R}^{d+1}_+)$. Note that
$$
\aligned
\int_{\mathbb{R}^{d+1}_+} \frac{\partial u}{\partial x_j} h(X)\, dX
&=2\int_{\mathbb{R}^d} g(y)\, dy \int_{\mathbb{R}^{d+1}_+}
\frac{\partial \widetilde{\Gamma}}{\partial x_j} (X, (y,0)) h(X)\, dX\\
&=-2
\int_{\mathbb{R}^d} g(y) v_j (y)\, dy,
\endaligned
$$
where
$$
v_j (y)=\int_{\mathbb{R}^{d+1}_+}
\widetilde{\Gamma}(X, (y,0)) \frac{\partial h}{\partial x_j}\, dX.
$$
Using 
$$
\| v_j\|_{L^2(B(0,R))}
\le C_R \| v_j\|_{W^{1,2}(T(0,R))} \le C_R \| h\|_{L^2(\mathbb{R}^{d+1}_+)},
$$
we may conclude
by duality that
$|\nabla u|\in L^2(T(0,R))$ for any $R>1$.
\end{proof}

\begin{definition}
{\rm Let $1<p<\infty$. We say the $L^p$ Neumann problem $(N)_p$ for $\mathbb{R}^{d+1}_+$
is solvable if for any $g\in L^\infty_c(\mathbb{R}^d)$, the weak solution
$u(X)=\int_{\mathbb{R}^d} N(X,y) f(y)dy$
satisfies $\| N(\nabla u)\|_p \le C\| g\|_p$.
}
\end{definition}

The solvability of $(N)_p$
yields the existence of weak solutions of the Neumann problem
with data in $L^p(\mathbb{R}^d)$.

\begin{thm}\label{Neumann-existence}
Let $1<p<\infty$. Suppose that $(N)_p$ for
$\mathcal{L}(u)=0$ in $\mathbb{R}^{d+1}_+$ is solvable.
Then for any $g\in L^p(\mathbb{R}^d)$, there exists
a weak solution $u$ to the Neumann problem with data $g$ such that
$\| N(\nabla u)\|_p \le C\| g\|_p$.
\end{thm}

\begin{proof}
Let $g\in L^p(\mathbb{R}^d)$.
Choose $g_k \in C_c(\mathbb{R}^d)$ such that $g_k\to g$
in $L^p(\mathbb{R}^d)$.
Let $u_k (X)=\int_{\mathbb{R}^d} N(X,y) g_k(y)dy$.
Then
$\| N(\nabla v_k)\|_p \le C\| g_k\|_p$
and $\| N(\nabla v_k -\nabla v_\ell)\|_p \le C\| g_k-g_\ell\|_p$.
By a limiting argument similar to that in the proof of Theorem \ref{regularity-existence},
there exists $u\in W^{1,2}_{loc} (\mathbb{R}^{d+1}_+)$
such that $\| N(\nabla u)\|_p \le C\| g\|_p$
and $\| N(\nabla u_k -\nabla u)\|_p \le C\| g_k -g\|_p$.
Note that the last inequality implies $\nabla u_k
\to \nabla u$ in $L^p(T(0,R))$ for any $R>1$.
It follows that $u$ is a weak solution to the Neumann problem
with data $g$.
\end{proof}

\begin{remark}\label{Neumann-representation-remark}
{\rm 
Suppose that $(N)_p$ for $\mathcal{L}(u)=0$
in $\mathbb{R}^{d+1}_+$ is solvable.
It follows from the proof of Theorem \ref{Neumann-existence} that
if $1<p\le 2$, the weak solution given by (\ref{Neumann-representation}) satisfies
$\| N(\nabla u)\|_p \le C\| g\|_p$.
}
\end{remark}

\begin{lemma}
Let $u\in W^{1,2}_{loc}(T(0,2R))$. Assume that
\begin{equation}\label{Neumann-data-zero}
\int_{\mathbb{R}^{d+1}_+}
a_{ij}\frac{\partial u}{\partial x_j}\cdot \frac{\partial \psi}{\partial x_i}\, dX=0
\end{equation}
for any $\psi \in C_0^1(B(0,2R)\times (-R,R))$.
Then
\begin{equation}\label{Holder-estimate}
|u(X)-u(Y)|\le C R \left(\frac{|X-Y|}{R}\right)^\beta
\left(\frac{1}{R^{d+1}}
\int_{T(0,2R)} |\nabla u|^2\, dX\right)^{1/2}
\end{equation}
for any $X,Y\in T(0,R)$, where $C>0$ and  $\beta\in (0,1]$ depend only on $d$ and $\mu$.
\end{lemma}

\begin{proof}
Estimate (\ref{Holder-estimate}) follows from the De Giorgi - Nash estimate
by a reflection argument.
\end{proof}

The next theorem addresses the question of uniqueness.

\begin{thm}\label{Neumann-uniqueness}
Let $u\in W^{1,2}_{loc}(\mathbb{R}^{d+1}_+)$ and $N(\nabla u)\in L^p(\mathbb{R}^d)$
for some $2\le p<\frac{d}{1-\beta}$, where $\beta$ is given by (\ref{Holder-estimate}).
Suppose that (\ref{Neumann-data-zero}) holds for any
$\psi\in C_0^1 (\mathbb{R}^{d+1})$.
Then $u$ is constant in $\mathbb{R}^{d+1}_+$.
\end{thm}

\begin{proof}
This follows readily from the estimate (\ref{Holder-estimate}).
\end{proof}

We end this section with a result on the implication
of $N(\nabla u)\in L^p(\mathbb{R}^d)$
on the Neumann data.

\begin{thm}\label{Neumann-problem-consequence-1}
Let $u\in W^{1,2}_{loc}(\mathbb{R}^{d+1}_+)$ be a weak solution of
$\mathcal{L}(u)=0$ in $\mathbb{R}^{d+1}_+$.
Suppose that $N(\nabla u)\in L^p(\mathbb{R}^d)$ for some $1<p<\infty$.
Then $u$ is a weak solution to the Neumann problem on $\mathbb{R}^{d+1}_+$
with data $g$ for some $g\in L^p(\mathbb{R}^d)$ .
Furthermore, $\| g\|_p\le C\| N(\nabla u)\|_p$ and
$$
-\frac{1}{\rho}\int_\rho^{2\rho}
a_{(d+1)j} (x,t)\frac{\partial u}{\partial x_j} dt \to g
\text{ weakly in } L^p(\mathbb{R}^d),
$$
as $\rho \to 0^+$.
\end{thm}

\begin{proof}
Let $\psi$ be a Lipschitz function on $\mathbb{R}^d$ with compact support. Define
\begin{equation}\label{definition of lambda}
\Lambda(\psi)
=\int_{\mathbb{R}^{d+1}_+}
a_{ij} \frac{\partial u}{\partial x_j}\frac{\partial\varphi}{\partial x_i}\, dX,
\end{equation}
where $\varphi$ is a Lipschitz function on $\mathbb{R}^{d+1}$ with compact support
 such that $\varphi(x,0)=\psi(x)$
for $x\in \mathbb{R}^d$. Since $N(\nabla u)\in L^p(\mathbb{R}^d)$, 
$|\nabla u|\in L^p(T(0,R))$ for any $R>1$. Thus the integral in 
(\ref{definition of lambda}) converges.
Also note that since $\mathcal{L}(u)=0$ in $\mathbb{R}^{d+1}_+$, the functional $\Lambda$
is well defined.

Next let $\varphi (x,t)=\psi(x)\eta(t)$ in (\ref{definition of lambda}), 
where $\eta$ is a continuous function such that $\eta (t)=1$ on $[0,\rho]$, $\eta(t)=0$ on
$[2\rho, \infty)$ and $\eta$ is linear on $[\rho, 2\rho]$.
This gives
\begin{equation}\label{consequence-1}
\Lambda (\psi)
=\lim_{\rho\to 0}
\left\{-\frac{1}{\rho}
\int_\rho^{2\rho}
\int_{\mathbb{R}^d} \sum_{j=1}^{d+1} a_{(d+1)j} \frac{\partial u}{\partial x_j}
\psi\, dxdt\right\},
\end{equation}
which, by H\"older inequality, implies that $|\Lambda (\psi)|\le C\| N(\nabla u)\|_p 
\| \psi\|_{p^\prime}$.
Hence, there exists $g\in L^p(\mathbb{R}^d)$ such that
$\| g\|_p \le C\, \|N(\nabla u)\|_p$ and 
$$
\int_{\mathbb{R}^{d+1}_+}
a_{ij}\frac{\partial u}{\partial x_j}
\frac{\partial \varphi}{\partial x_i}
\, dX
=\Lambda(\psi)
=\int_{\mathbb{R}^d} g\psi\, dx
=\int_{\mathbb{R}^d} g\varphi\, dx.
$$
\end{proof}

\section{$L^2$ estimates for the Neumann problem}

To solve the $L^2$ Neumann problem,
as in the cases of the Dirichlet and regularity problems,
we need to impose a local solvability condition: for $x_0\in \mathbb{R}^d$,
\begin{equation}
\label{local-Neumann-condition}
\int_{B(x_0,1)} |N_1(\nabla u)|^2\, dx
\le C_2 
\left\{
\int_{B(x_0,2)} |g|^2\, dx
+\int_0^2\int_{B(x_0,4)} |\nabla u|^2\, dxdt\right\},
\end{equation}
whenever  $u \in W^{1,2}(T(x_0,4))$
 satisfies (\ref{variational-formulation})
with $g\in C(B(x_0,4))$ 
for all $\varphi\in C_0^1 (B(x_0,4)\times (-4,4))$.

The goal of this section is to prove the following.

\begin{thm}\label{Neumann-theorem-2}
Let $\mathcal{L}=-\text{div}(A\nabla )$ with smooth coefficients
satisfying (\ref{real-symmetric}), (\ref{ellipticity}) and
(\ref{periodic-in-t}).
Also assume that $\mathcal{L}$
satisfies conditions
(\ref{local-Dirichlet-condition}) and (\ref{local-Neumann-condition}).
Then, given any $g\in L^2(\mathbb{R}^{d})$, there exists a
$u\in W^{1,2}_{loc}(\mathbb{R}^{d+1}_+)$, unique up to constants, 
satisfying 
$N(\nabla u)\in L^2(\mathbb{R}^d)$ and
(\ref{variational-formulation}).
Moreover the solution $u$ satisfies the estimate
$\| N(\nabla u)\|_2 \le C\, \| g\|_2$,
where $C$ depends only on $d$, $\mu$ and the constants $C_0$ 
in (\ref{local-Dirichlet-condition}) and $C_2$ in (\ref{local-Neumann-condition}).
\end{thm}

In view of Theorems \ref{Neumann-existence} and \ref{Neumann-uniqueness},
it suffices to prove the following.

\begin{thm}\label{Neumann-L-2-solvability}
Under the same conditions on $\mathcal{L}$ as in Theorem \ref{Neumann-theorem-2},
the $L^2$ Neumann problem for $\mathcal{L}(u)=0$ in $\mathbb{R}^{d+1}_+$
is solvable.
\end{thm}

\begin{proof}
Let $g\in C_c(\mathbb{R}^d)$ and $u$ be given by
(\ref{Neumann-representation}).
It follows from  Lemma \ref{Neumann-function-lemma} that $u\in W^{1,2}(T(0,R))$
for any $R>1$.
This allows us to deduce from the condition (\ref{local-Neumann-condition})
by an integration in $x_0$ over $\mathbb{R}^d$ that
\begin{equation}\label{Neumann-existence-1}
\int_{\mathbb{R}^d}|N_4 (\nabla u)|^2\, dx
\le C\left\{ \int_{\mathbb{R}^d}|g|^2\, dx
+\int_0^6 \int_{\mathbb{R}^d} |\nabla u|^2\, dxdt\right\}.
\end{equation}
Note that by the pointwise estimates on $N(X,Y)$, we have $u(X)=O(|X|^{1-d})$ as $|X|\to \infty$. 
By Cacciopoli's inequality,
this implies that
$$
\iint_{T(0,3R)\setminus T(0,R)}
|\nabla u|^2\, dxdt
\to 0 \text{ as } R\to \infty.
$$
In view of (\ref{key-Neumann-estimate-1}), we obtain
\begin{equation}\label{Neumann-existence-2}
\int_0^6 \int_{\mathbb{R}^d}
|\nabla u|^2\, dxdt
\le C\, \int_{\mathbb{R}^d} |g|^2\, dx.
\end{equation}
By (\ref{Neumann-existence-1}), this leads to
$\| N_4(\nabla u)\|_2 \le C\| g\|_2$.

Next, using the $L^2$ bound of $N_4(\nabla u)$ above 
and Lemma \ref{nontangential-maximal-consequence}, we see that
$u\to f$ n.t. on $\mathbb{R}^d$
and $\| \nabla_x f\|_2 \le C\| N_4(\nabla u)\|_2\le C\| g\|_2$.
Thus, by Remark \ref{maximal-function-estimate-remark},
\begin{equation}\label{Neumann-existence-3}
\aligned
\int_{\mathbb{R}^d}
|N(\nabla u)|^2\, dx
&\le C\left\{
\int_{\mathbb{R}^d} |\nabla_x u(x,0)|^2\, dx
+\int_{\mathbb{R}^d} |N_4 (\nabla u)|^2\, dx
+\int_{\mathbb{R}^d}
|\big( Q(u)\big)^*|^2\, dx\right\}\\
&\le C\left\{
\int_{\mathbb{R}^d} |g|^2\, dx
+\int_{\mathbb{R}^d}
|\big( Q(u)\big)^*|^2\, dx\right\}.
\endaligned
\end{equation}

Finally we note that by (\ref{decay of Neumann function}), $|Q(u) (X)|\le C|X|^{1-d-\alpha}$
for $|X|$ large. This implies that $\big(Q(u)\big)^*\in L^2(\mathbb{R}^d)$.
Since the $L^2$ Dirichlet problem for $\mathcal{L}(u)=0$ in $\mathbb{R}^{d+1}_+$
is solvable, we obtain
$$
\int_{\mathbb{R}^d} |\big(Q(u)\big)^*|^2\, dx
\le C\int_{\mathbb{R}^d} |Q(u)|^2\, dx
\le C\int_0^1 \int_{\mathbb{R}^d} |\nabla u(x,t)|^2\, dxdt
\le C\int_{\mathbb{R}^d} |g|^2\, dx
$$
where we used (\ref{Neumann-existence-2}) in the last step.
In view of (\ref{Neumann-existence-3}), we have proved that
$\|N(\nabla u)\|_2 \le C\| g\|_2$.
\end{proof}

We end this section with a localization theorem.

\begin{thm}\label{Neumann-localization}
Let $\mathcal{L}=-\text{div}(A\nabla)$ with coefficient matrix $A(X)$ satisfying 
(\ref{real-symmetric})-(\ref{ellipticity}).
Suppose that $(N)_2$ and $(R)_2$ for
$\mathcal{L}(u)=0$ in $\mathbb{R}^{d+1}_+$ is solvable.
Let $u\in W^{1,2}(T(x_0, 6r))$ be a weak solution in $T(x_0,6r)$
such that $N_r(\nabla u)\in L^2(B(x_0,4r))$ and $\frac{\partial u}{\partial \nu}
=g$ on $B(x_0,4r)$.
Then
\begin{equation}
\int_{B(x_0,r)} |N_r(\nabla u)|^2\, dx
\le C\int_{B(x_0,3r)} |g|^2\, dx
+\frac{C}{r}
\int_{T(x_0,3r)} |\nabla u|^2\, dX.
\end{equation}
\end{thm}

\begin{proof}
Let $\varphi\in C_0^\infty (B(x_0,3r)\times (-3r,3r))$ such that
$\varphi=1$ in $T(x_0,2r)$ and $|\nabla \varphi|\le C/r$.
Let $w=u-\beta$ where $\beta$ is the average of $u$ over $T(x_0,3r)$.
Then for $X\in T(x_0,2r)$,
\begin{equation}
\aligned
w(X)
&=\int_{\mathbb{R}^{d+1}_+}
a_{ij} \frac{\partial }{\partial y_j} \big\{ N(X,Y)\big\}
\cdot \frac{\partial (w\varphi)}{\partial y_i}\, dY\\
&
=\int_{\mathbb{R}^d}
N(X,Y) g\varphi\, dY
-\int_{\mathbb{R}^{d+1}_+}
N(X,Y) a_{ij} \frac{\partial w}{\partial y_j}
\cdot \frac{\partial \varphi}{\partial y_i} \, dY\\
&\qquad\qquad
+\int_{\mathbb{R}^{d+1}_+}
a_{ij} \frac{\partial }{\partial y_j} \big\{ N(X,Y)\big\}
w
\cdot
\frac{\partial \varphi}{\partial y_i}\, dY
\endaligned
\end{equation}
With this representation formula, the rest of the proof 
similar to that of Theorem 6.10 in \cite{KP-1993}.
We omit the details.
\end{proof}

\begin{remark}\label{Neumann-localization-remark}
{\rm
Let $u\in W^{1,2}_{loc}(T(x_0,6r))$ be a weak solution in $T(x_0,6r)$ such that
$N(\nabla u)\in L^2(B(x_0,4r))$ and $\frac{\partial u}{\partial\nu} =0$
on $B(x_0,6r)$.
Suppose that $(N)_2$ and $(R)_2$ for $\mathcal{L}(u)
=0$ in $\mathbb{R}^{d+1}_+$ are solvable.
Let $E$ be the average of $u$ over $T(x_0,4r)$.
It follows from Theorem \ref{Neumann-localization} that
$$
\aligned
\bigg(
\frac{1}{r^{d}}
& \int_{B(x_0,r)} |N_r(\nabla u)|^2\, dx\bigg)^{1/2}
\le C\left(\frac{1}{r^{d+1}}
\int_{T(x_0, 3r)} |\nabla u|^2 \, dX\right)^{1/2}\\
& \le \frac{C}{r} 
\left(\frac{1}{r^{d+1}}
\int_{T(x_0, 3.1r)} |u-E|^2 \, dX\right)^{1/2}
\le \frac{C}{r^{d+2}}
\int_{T(x_0,3.2r)} |u-E |\, dX\\
&
\le \frac{C}{r^{d+1}}
\int_{T(x_0,3.2r)} |\nabla u|\, dX \le
\frac{C}{r^d}
\int_{B(x_0,4r)} |N(\nabla u)|\, dx,
\endaligned
$$
where we have used Cacciopoli's inequality, boundary regularity of weak solutions, 
and Poincar\'e inequality.
This yields the weak reverse H\"older inequality,
\begin{equation}\label{reverse-Holder-Neumann-0}
\left\{
\frac{1}{r^{d}}
\int_{B(x_0,r)} |N(\nabla u)|^2\, dx\right\}^{1/2}
\le \frac{C}{r^d}
\int_{B(x_0,4r)} |N(\nabla u)|\, dx.
\end{equation}
By the self-improving property of the weak reverse H\"older inequality,
there exists $p>2$, depending only on $d$ and the constant $C$ in 
(\ref{reverse-Holder-Neumann-0}),  such that
\begin{equation}\label{Neumann-reverse-Holder}
\left\{
\frac{1}{r^{d}}
\int_{B(x_0,r)} |N(\nabla u)|^p\, dx\right\}^{1/p}
\le C \left\{\frac{1}{r^d}
\int_{B(x_0,2r)} |N(\nabla u)|^2\, dx\right\}^{1/2}.
\end{equation}
}
\end{remark}

\section{$L^p$ estimates for the Neumann problem}

In this section we study the $L^p$ Neumann problem in $\mathbb{R}^{d+1}_+$.

\begin{thm}\label{Neumann-solvability-p}
Let $\mathcal{L}=-\text{div}(A\nabla)$ with coefficients satisfying 
(\ref{real-symmetric})-(\ref{ellipticity}).
Suppose that $(N)_2$ and $(R)_2$ for $\mathcal{L}(u)=0$
in $\mathbb{R}^{d+1}_+$ are both solvable.
Then $(N)_p$ for $\mathcal{L}(u)=0$ in $\mathbb{R}^{d+1}_+$
is solvable for $1<p<2+\delta$.
\end{thm}

\begin{proof}
In the case of the unit ball (or a star-shaped Lipschitz domain), it was proved in
\cite{KP-1993} that the solvability of $(N)_{p_0}$ and $(R)_{p_0}$
implies the solvability of $(N)_p$ for
any $1<p<p_0+\delta$.
We follow the approach in \cite{Dahlberg-Kenig-1987} and 
\cite{KP-1993} to treat the case $1<p<2$.
For $f\in L^2(\mathbb{R}^d)$ with compact support, let
$T(f)= N(\nabla u)$, where $u$ is given by (\ref{Neumann-representation}).
Since $(N)_2$ is solvable, $T$ is bounded on $L^2(\mathbb{R}^d)$.
With Theorem \ref{Neumann-localization} and pointwise estimates of the Neumann function
at our disposal, one may establish 
$\|T(f)\|_1 \le C$, where $f$ is an $H^1_{at}$ atom.
The $L^p$ boundedness of $T$ then follows by interpolation.

We use Theorem \ref{real-variable-theorem} to treat the range $p>2$.
Let $f\in L^2(\mathbb{R}^d)$ with compact support and $F=T(f)=N(\nabla u)$,
where $u$ is given by (\ref{Neumann-representation}).
For each ball $B$ in $\mathbb{R}^d$, we choose $F_B= T(f\chi_{6B})$ and
$R_B=T(f\chi_{\mathbb{R}^d\setminus 6B})$.
Then $|F|\le F_B +R_B$.
By the $L^2$ boundedness of $T$,
$$
\int_{B} |F_B|^2\, dx 
\le
C\int_{6B} |f|^2\, dx.
$$
To estimate $R_B$, we use the weak reverse H\"older inequality (\ref{Neumann-reverse-Holder}).
This gives
$$
\aligned
\left(\frac{1}{r^d}
\int_{B} |R_B|^p\, dx\right)^{1/p}
&\le
C\left(\frac{1}{r^d}
\int_{2B}
|R_B|^2\, dx\right)^{1/2}\\
&\le C \left(\frac{1}{r^d}
\int_{2B}
|T(f)|^2\, dx\right)^{1/2}
+C\left(\frac{1}{r^d}
\int_{6B} |f|^2\, dx \right)^{1/2}.
\endaligned
$$
Thus, by Theorem \ref{real-variable-theorem}, we obtain $\|T(f)\|_q
\le C\| f\|_q$ for any $2<q<p$.

\end{proof}

The next theorem deals with the question of uniqueness for $1<p<2$.

\begin{thm}
Let $\mathcal{L}(u)=-\text{div}(A\nabla )$ with coefficients satisfying
(\ref{real-symmetric})-(\ref{ellipticity}).
Let $1<p<2$. 
Assume that $(R)_p$ and $(N)_2$
for $\mathcal{L}(u)=0$ in $\mathbb{R}^{d+1}_+$ are both solvable.
Suppose that $N(\nabla u)\in L^p(\mathbb{R}^d)$ and $u\in W^{1,2}_{loc}
(\mathbb{R}^{d+1}_+)$ satisfies
(\ref{Neumann-data-zero}).
Then $u$ is constant in $\mathbb{R}^{d+1}_+$.
\end{thm}

\begin{proof}
Since $N(\nabla u)\in L^p(\mathbb{R}^d)$, $u= f$ n.t. on $\mathbb{R}^d$
for some $f\in \dot{W}^{1,p}(\mathbb{R}^d)$ by Lemma \ref{nontangential-maximal-consequence}.
Thus $u$ is the solution of the $L^p$ regularity problem with boundary data
$f$. Choose $f_k \in C_0^1(\mathbb{R}^d)$ such that $\nabla_x f_k \to \nabla_x f$
in $L^p(\mathbb{R}^d)$.
Let $u_k$ be the solution of the classical Dirichlet problem with data $f_k$.
By Theorem \ref{Neumann-problem-consequence-1}, $u_k$ is also a weak solution
to the Neumann problem on $\mathbb{R}^{d+1}_+$ with data $g_k\in L^p(\mathbb{R}^d)$.
Since $\|g_k-g_\ell\|_p \le C\|N(\nabla u_k-\nabla u_\ell)\|_p\to 0$ as $k, \ell\to\infty$, we 
may deduce that $g_k \to g$ in $L^p(\mathbb{R}^d)$.

Next we note that $\|N(\nabla u_k -\nabla u)\|_p \le C\| \nabla_x (f_k -f)\|_p$.
Thus $\nabla u_k \to \nabla u$ in $L^1(T(0,R))$ for any $R>1$.
It follows that $\int_{\mathbb{R}^d} g\varphi dx=0$ for any $\varphi\in C_0^1(\mathbb{R}^{d+1})$.
Consequently, we obtain $g=0$.

Finally since $\| N(\nabla u_k)\|_2 \le C\|\nabla_x f_k\|_2<\infty$,
$u_k$ is a solution to the $L^2$ Neumann problem with data $g_k$. 
By Remark \ref{Neumann-representation-remark}, $u_k$ is given by
$
u_k =\beta_k +\int_{\mathbb{R}^d} N(X,y) g_k (y)\, dy
$
for some constant $\beta_k$, if $d\ge 3$. If $d=2$, one needs to replace
$N(X,y)$ by $N(X,y)-N(X_0,y)$.
In any case, because $g_k\to 0$ in $L^p(\mathbb{R}^d)$,
we have $u_k(x)-\beta_k\to 0$ uniformly on any compact subset of $\mathbb{R}^{d+1}_+$.
This, together with the fact that $\nabla u_k\to \nabla u$ in $L^1(T(0,R)$,
implies that $\nabla u=0$ a.e. in $\mathbb{R}^{d+1}_+$.
Hence $u$ is constant in $\mathbb{R}^{d+1}_+$.
\end{proof}

\section{Proof of Theorems \ref{Dirichlet-theorem-Lipschitz},
 \ref{Neumann-theorem-Lipschitz} and \ref{regularity-theorem-Lipschitz}}

In this section
 we give the proof of Theorems \ref{Dirichlet-theorem-Lipschitz},
 \ref{Neumann-theorem-Lipschitz} and \ref{regularity-theorem-Lipschitz}
stated in the Introduction.
Let
$$
D=\big\{ (x,t): x\in \mathbb{R}^d \text{ and }\ t>\psi (x)\big\},
$$
where $\psi:\mathbb{R}^d \to\mathbb{R}$ is a Lipschitz function with Lipschitz
constant $\|\nabla_x \psi\|_\infty$.
Consider the bi-Lipschitzian mapping $\Phi : D \to \mathbb{R}^{d+1}_+$ given by
$
Y=\Phi(X)=\Phi(x,t)=(x,t-\psi(x)).
$
Let $\Phi^\prime$ denote the Jacobian matrix of $\Phi$ and
$\widetilde{A}(Y)=(\widetilde{a}_{ij}(Y))=\Phi^\prime (X) A (X) (\Phi^\prime (X))^t$,
where $X=\Phi^{-1} (Y)=\Psi(Y)$.
Note that
$$
\int_D a_{ij} (X) \frac{\partial u}{\partial x_j} \frac{\partial \varphi}{\partial x_i}\, dX
=\int_{\mathbb{R}^{d+1}_+}
\widetilde{a}_{ij} (Y)
\frac{\partial \widetilde{u}}{\partial y_j} \frac{\partial \widetilde{\varphi}}{\partial y_i}\, dY,
$$
where $\widetilde{u}(Y)=u(\Psi(Y))$
and $\widetilde{\varphi} (Y)=\varphi(\Psi (Y))$.
This shows that $\mathcal{L}(u)=-\text{div} (A\nabla u)=0$ in $D$ if and only if
$\widetilde{\mathcal{L}}(\widetilde{u})
=-\text{div} (\widetilde{A}\nabla \widetilde{u})=0 $ in $ \mathbb{R}^{d+1}_+$.
Consequently, the $L^p$ Dirichlet problem for $\mathcal{L}(u)=0$
in $D$ is solvable if and only if the $L^p$ Dirichlet problem for
$\widetilde{\mathcal{L}}(u)=0$ in $\mathbb{R}^{d+1}_+$
is solvable.
The same can be said for the $L^p$ regularity and Neumann problems.

Note that $\Phi^\prime (\Psi(Y))$ is independent of $s=y_{d+1}$ and
$A(\Psi(Y))=A(y,s+\psi(y))$. It follows that
$\widetilde{A}(y,s+1)=\widetilde{A}(y,s)$ if $A(x,t+1)=A(x,t)$.
Also, since
$$
\widetilde{A}(y,s_1)-\widetilde{A}(y,s_2)
=\Phi^\prime (X)
\big\{ A(y,s_1+\psi(y))-A(y,s_2+\psi(y))\big\} (\Phi^\prime (X))^t,
$$
there exist $c_1, c_2>0$ depending only on $d$ and $\|\nabla_x\psi\|_\infty$
such that $c_1 \eta (\rho) \le \widetilde{\eta}(\rho)
\le c_2 \eta (\rho)$, where $\eta (\rho)$ is given by (\ref{eta-function}) and 
$\widetilde{\eta} (\rho)$ is defined in the same manner, using $\widetilde{A}$.
Therefore the elliptic operator $\widetilde{\mathcal{L}}$ satisfies the 
same assumptions
as those imposed on $\mathcal{L}$ in Theorem \ref{Dirichlet-theorem-Lipschitz}.
As a result, it suffices to prove Theorems \ref{Dirichlet-theorem-Lipschitz},
\ref{Neumann-theorem-Lipschitz} and \ref{regularity-theorem-Lipschitz}
in the case $D=\mathbb{R}^{d+1}_+$.

Next we reduce the general case to the case of smooth coefficients.

\begin{lemma}\label{regularity-approximation-lemma}
Let $\mathcal{L}_k =-\text{div}( A_k\nabla )$ and
$\mathcal{L}=-\text{div}(A \nabla )$ with coefficient matrices satisfying
(\ref{real-symmetric})-(\ref{ellipticity}).
Suppose that $A_k(X)\to A(X)$ for a.e. $X\in \mathbb{R}^{d+1}$ as $k\to\infty$.
Let $u_k$ and $u$ be solutions to the classical Dirichlet problem for
$\mathcal{L}_k$ and $\mathcal{L}$ respectively in $\mathbb{R}^{d+1}_+$ 
with boundary data $f\in C_0^1(\mathbb{R}^d)$.
Then, as $k\to\infty$,  $u_k\to u$ uniformly in $T(0,R)$ for any $R>1$
 and $\nabla u_k \to \nabla u$ in $L^2(\mathbb{R}^{d+1}_+)$. Consequently,
$$
\aligned
\|(u)^*\|_p & \le \liminf_{k\to \infty} \| (u_k)^*\|_p,\\
\|N(\nabla u)\|_p & \le \liminf_{k\to \infty} \| N(\nabla u_k)\|_p.
\endaligned
$$
for any $p\ge 1$.
\end{lemma}

\begin{proof}
We may assume that $f\in C_0^1(\mathbb{R}^{d+1})$. Then
$$
\int_{\mathbb{R}^{d+1}_+}
A\nabla (u-f)\cdot \nabla g\, dX =-\int_{\mathbb{R}^{d+1}_+} A
\nabla f\cdot \nabla g\, dX
$$
for any $g\in C_0^1(\mathbb{R}^{d+1}_+)$. 
Recall that $u(X)=\int_{\mathbb{R}^d} f d\omega^X= O(|X|^{1-d-\delta})$,
as $|X|\to \infty$.
By taking $g(X)=(u(X)-f(X))\varphi (X/R)$ where $\varphi\in C_0^\infty (\mathbb{R}^{d+1})$,
$\varphi(X)=1$ for $|X|<1$ and $\varphi(X)=0$ for $|X|\ge 2$,
we may deduce that $\|\nabla u\|_{L^2(\mathbb{R}^{d+1}_+)}
\le C\| \nabla f\|_{L^2(\mathbb{R}^{d+1}_+)}$.
Similarly, using
$$
\int_{\mathbb{R}^{d+1}_+}
A_k \nabla (u_k-u)\cdot  \nabla g\, dX
=-\int_{\mathbb{R}^{d+1}_+}
(A_k-A) \nabla u\cdot \nabla g\, dX,
$$
for any $g\in C_0^1(\mathbb{R}^{d+1}_+)$,
we obtain
$$
\int_{\mathbb{R}^{d+1}_+}
|\nabla (u_k -u)|^2\, dX
\le C\int_{\mathbb{R}^{d+1}_+}
|A_k-A|^2 |\nabla u|^2\, dX.
$$
It follows from the dominated convergence theorem that $\nabla u_k \to \nabla u$
in $L^2(\mathbb{R}^{d+1}_+)$.

To see $u_k\to u$ uniformly in $T(0,R)$, we first note that
$u_k\to u$ in $L^2(T(0,R))$.
This follows from the Poincar\'e inequality
$$
\int_{T(0,R)}
|u_k-u|^2\, dX \le CR^2 \int_{T(0,R)}
|\nabla (u_k -u)|^2\, dX,
$$
as $u_k-u=0$ on $\mathbb{R}^d$.
However, by De Giorgi -Nash estimates, the sequence $\{ u_j\}$ is  equiv-continuous
on $T(0,R)$.
This implies that any subsequence of $\{ u_k\}$
contains a subsequence which converges uniformly on $T(0,R)$ to $u$.
It follows that $\{ u_k\}$ converges uniformly on $T(0,R)$ to $u$.
\end{proof}

\begin{thm}\label{regularity-approximation-theorem}
Let $\mathcal{L}_k=-\text{div}(A_k\nabla )$ and $\mathcal{L}
=-\text{div}(A\nabla)$ with coefficient matrices satisfying the same conditions
as in Lemma \ref{regularity-approximation-lemma}.
Suppose that $(R)_2$
for $\mathcal{L}_k (u_k)=0$ in $\mathbb{R}^{d+1}_+$ is solvable
with uniform estimate $\|N(\nabla u_k)\|_2 \le C_3\|\nabla_x u_k\|_2$.
Then $(R)_p$
for $\mathcal{L}(u)=0$ in $\mathbb{R}^{d+1}_+$
is solvable for $1<p<p_0$, where $p_0>2$ depends only on $d$, $\mu$ and $C_3$.
Consequently, $(D)_q$ is solvable for $p_0^\prime<q<\infty$.
\end{thm}

\begin{proof}
By Theorem \ref{regularity-theorem-p} and Theorem \ref{regularity-imply-Dirichlet},
we only need to show that $(R)_2$ for 
$\mathcal{L}(u)=0$ in $\mathbb{R}^{d+1}_+$ is solvable.
To this end 
let $u_k$ and $u$ be the solutions to the classical Dirichlet problem
for $\mathcal{L}(u)=0$ and $\mathcal{L}_k (u_k)=0$
in $\mathbb{R}^{d+1}_+$ with data $f\in C_0^1 (\mathbb{R}^{d})$, respectively.
It follows from Lemma \ref{regularity-approximation-lemma} 
and the uniform estimate $\| N(\nabla u_k)\|_2 \le C_3\|\nabla_x u_k\|_2$
that
$$
\| N(\nabla u)\|_2 
\le \liminf_{k\to\infty} 
\|N(\nabla u_k)\|_2
\le C_3\|\nabla_x f\|_2.
$$
\end{proof}

A similar result holds for the $L^2$ Neumann problem.

\begin{lemma}\label{Neumann-approximation-lemma}
Let $\mathcal{L}_k=-\text{div}(A_k\nabla )$ and $\mathcal{L}
=-\text{div}(A\nabla)$ with coefficient matrices satisfying the same conditions
as in Lemma \ref{regularity-approximation-lemma}.
Let 
$$
\aligned
u_k(X) &=\int_{\mathbb{R}^d} N_k (X, (y,0)) f(y)\, dy,\\
u(X) &=\int_{\mathbb{R}^d} N (X, (y,0)) f(y)\, dy,
\endaligned
$$
where $f\in L_c^\infty(\mathbb{R}^d)$ and
$N_k(X,Y)$, $N(X,Y)$ are the Neumann functions constructed in Section 7
for $\mathcal{L}_k$, $\mathcal{L}$ respectively.
Then $\nabla u_k \to \nabla u$ in $L^2(\mathbb{R}^{d+1}_+)$ and consequently,
$$
\|N(\nabla u)\|_p \le \liminf_{j\to\infty}
\|N(\nabla u_k)\|_p
$$
for any $p\ge 1$.
\end{lemma}

\begin{proof}
The proof is similar to that of Lemma \ref{regularity-approximation-lemma}.
Suppose that supp$(f)\subset B(0,r_0)$. Starting with
$$
\int_{\mathbb{R}^{d+1}_+} A\nabla u\cdot \nabla g \, dX=\int_{\mathbb{R}^d} fg\, dx
\quad \text{ for } \ g\in C_0^\infty (\mathbb{R}^{d+1})
$$
and using $u(X)=O(|X|^{1-d})$ as $|X|\to\infty$, we may deduce that
$$
\int_{\mathbb{R}^{d+1}_+} |\nabla u|^2\, dX \le Cr_0^{d+1} \| f\|_\infty^2.
$$
Similarly, using
$$
\int_{\mathbb{R}^{d+1}_+}
A_k \nabla (u_k-u)\cdot \nabla g\, dX
=-\int_{\mathbb{R}^{d+1}_+}
(A_k -A)\nabla u\cdot \nabla g\, dX
\quad \text{ for } \ g\in C_0^\infty (\mathbb{R}^{d+1}),
$$
we obtain
$$
\int_{\mathbb{R}^{d+1}_+}
|\nabla (u_k-u)|^2\, dX
\le C\int_{\mathbb{R}^{d+1}_+}
|A_k-A|^2|\nabla u|^2\, dX.
$$
Thus $\nabla u_k \to \nabla u $ in $L^2(\mathbb{R}^{d+1}_+)$
by the dominated convergence theorem.
\end{proof}

\begin{thm}\label{Neumann-approximation-theorem}
Let $\mathcal{L}_k=-\text{div}(A_k\nabla )$ and $\mathcal{L}
=-\text{div}(A\nabla)$ with coefficient matrices satisfying the same conditions
as in Lemma \ref{regularity-approximation-lemma}.
Suppose that $(N)_2$
for $\mathcal{L}_k (u_k)=0$ in $\mathbb{R}^{d+1}_+$
is solvable with uniform estimate
$\|N(\nabla u_k)\|_2 \le C_4\| \partial u_k/\partial \nu_k\|_2$.
Then $(N)_2$ for $\mathcal{L}(u)=0$
in $\mathbb{R}^{d+1}_+$ is solvable.
If in addition, $(R)_2$ is solvable
for $\mathcal{L}_k (u_k)=0$ with uniform estimate
$\|N(\nabla u_k)\|_2 \le C_4 \|\nabla_x u_k\|_2$, then both
$(N)_p$ and $(R)_p$ for $\mathcal{L}(u)=0$ in $\mathbb{R}^{d+1}_+$
are solvable for $1<p<2+\delta$, where $\delta>0$
depends only on $n$, $\mu$ and $C_4$. Furthermore, the unique solutions
of $(N)_p$ and $(R)_p$ satisfy
$\|N(\nabla u)\|_p \le C\| \partial u/\partial \nu \|_p$
and $\| N(\nabla u)\|_p \le C\| \nabla_x u\|_p$ respectively
with constant $C$ depending only on $d$, $p$, $\mu$ and $C_4$.
\end{thm}

\begin{proof}
This follows readily from 
Lemma \ref{Neumann-approximation-lemma}, Theorems \ref{regularity-approximation-theorem}
and \ref{Neumann-solvability-p}.
\end{proof}

Now given a $(d+1)\times (d+1)$ matrix $A(X)$ on $\mathbb{R}^{d+1}$ that satisfies
conditions (\ref{real-symmetric}), (\ref{ellipticity}),
(\ref{periodic-in-t}) and (\ref{smoothness-condition-in-t}),
we define
$$
A_k(X)= \int_{\mathbb{R}^{d+1}}
A(X-Y)\varphi_k (Y)\, dY
$$
where $\{ \varphi_k (Y)\}$ is a standard approximation of identity.
Then $A_k\in C^\infty (\mathbb{R}^{d+1})$ and
$A_k(X)\to A(X)$ for a.e. $X\in \mathbb{R}^{d+1}$.
Furthermore, $A_k(X)$
satisfies the same conditions (\ref{real-symmetric}), (\ref{ellipticity}),
(\ref{periodic-in-t}) and (\ref{smoothness-condition-in-t})
with the same $\mu$ and $\eta_k(\rho)\le \eta(\rho)$, where $\eta_k(\rho)$
denotes the modulus of continuity for $A_k$ in the $t$ variable.
In view of Theorem \ref{Neumann-approximation-theorem},
we have reduced the proofs of Theorems \ref{Dirichlet-theorem-Lipschitz},
\ref{Neumann-theorem-Lipschitz}
and \ref{regularity-theorem-Lipschitz}
to the proof of the following theorem.

\begin{thm}\label{special-case-theorem}
Theorems \ref{Dirichlet-theorem-Lipschitz},
\ref{Neumann-theorem-Lipschitz}
and \ref{regularity-theorem-Lipschitz} hold
in the case where $D=\mathbb{R}^{d+1}_+$, $p=2$, and the coefficient
matrix $A(X)$ is $C^\infty$ and
satisfies conditions (\ref{real-symmetric}), (\ref{ellipticity}),
(\ref{periodic-in-t}) and (\ref{smoothness-condition-in-t}).
\end{thm}
 
\begin{proof}
In view of Theorems \ref{Dahlberg-theorem}, \ref{regularity-theorem} and \ref{Neumann-theorem-2},
it suffices to show that the square Dini condition
(\ref{smoothness-condition-in-t}) implies the local solvability conditions
(\ref{local-regularity-condition}) and (\ref{local-Neumann-condition}).
To this end, we fix $x_0\in \mathbb{R}^d$ and choose $\varphi_1\in C_0^\infty(\mathbb{R}^d)$
and $\varphi_2\in C_0^\infty(\mathbb{R})$ so that $0\le \varphi_1\le 1$,
$\varphi_1 (x) =1$ in $B(x_0,8)$, $\varphi_1 (x)=0$ outside of $B(x_0,9)$
and $0\le \varphi_2 \le 1$, $\varphi_2 (t) =1$ in $[-8, 8]$, $\varphi_2 (t)=0$ outside of $[-9,9]$.
We consider the elliptic operator
$L_1 =-\text{div}(A_1\nabla)$ in the star-like Lipschitz domain $\Omega_0=T(x_0,20)$, where
\begin{equation}\label{definition-of-A-1} 
A_1(x,t)=\varphi_2 (t) \big[ \varphi_1 (x) A(x,t) + (1-\varphi_1 (x))I\big] +(1-\varphi_2(t)) I
\end{equation}
and 
$I$ denotes the $(d+1)\times (d+1)$ identity matrix.
Note that $A_1(x,t)=A(x,t)$ for $(x,t)\in T(x_0,8)$ and
thus $\mathcal{L}(u)=0$ in $T(x_0,r)$ implies $L_1(u)=0$ in $T(x_0,r)$, if $0<r<8$.
We claim that under the condition (\ref{smoothness-condition-in-t}),
both $(N)_2$ and $(R)_2$ for $L_1(u)=0$ in $\Omega_0$ are solvable and
the bounding constants in the nontangential maximal function estimates depend
only on $d$, $\mu$ and $\eta (t)$ in (\ref{eta-function}).
By the localization results for the $L^p$ 
Neumann and regularity problems in star-like Lipschitz domains
(see Theorems 5.19 and 6.10 in \cite{KP-1993}), this gives estimates
(\ref{local-regularity-condition}) and (\ref{local-Neumann-condition}).
We should point out that Theorems 5.19 and 6.10 in \cite{KP-1993}
were stated for solutions of $L_1 (u)=0$ in $\Omega_0$.
However the same arguments in their proofs apply to local solutions of
$L_1(u)=\mathcal{L}(u)=0$ in $T(x_0,r)$
for $0<r<8$.

To prove the claim we introduce another elliptic operator
$L_2=-\text{div}(A_2\nabla )$ in $\Omega_0$, where
\begin{equation}\label{definition-of-A-2}
 A_2(x,t)=\varphi_2 (t) \big[ \varphi_1 (x) A(x,0) + (1-\varphi_1 (x))I\big]
+(1-\varphi_2 (t))I.
\end{equation}
Let
\begin{equation}\label{definition-of-e}
\varep^* (X)=\sup \big\{ |A_1(Y)-A_2(Y)|:\
Y\in B(X, \delta(X)/2)\big\},
\end{equation}
where $\delta(X)=\text{dist}(X, \partial\Omega_0)$.
Since
\begin{equation}\label{special-case-3}
A_1(x,t)-A_2 (x,t) =\varphi_1(x) \varphi_2 (t) \big( A(x,t)-A(x,0)\big),
\end{equation}
it follows that $\varep^*(X)=0$ if $|X-P|\le 1$ for some 
$P\in \partial\Omega_0 \setminus \{ (y,0): |y|<10\}$.
It is also easy to see from (\ref{special-case-3}) that if $X=(x,t)$ with $|x|\le 11$ and $t\in (0,1)$, then
$\varep^*(X)\le \eta (2t)$, where $\eta(\cdot)$ is defined by (\ref{eta-function}).
Therefore, for $P\in \partial\Omega_0$ and $r\in (0,1)$,
\begin{equation}\label{definition-of-h}
h(P,r):=\frac{1}{r^d}
\int_{B(P,r)\cap \Omega_0}
\frac{ \big(\varep^* (X)\big)^2}{\delta(X)} \, dX
\le C\int_0^{cr} \frac{\big(\eta (t)\big)^2}{t}\, dt.
\end{equation}
It follows from the condition (\ref{smoothness-condition-in-t}) that
$$
\sup \big\{ h(P,r): P\in \partial\Omega_0\big\} \to 0 
\text{ as } r\to 0.
$$
Hence, by Theorem 2.2 in \cite{KP-1995}, the solvability of $(N)_2$ and $(R)_2$ for
$L_1 (u)=0$ in $\Omega_0$ is equivalent to
that of $(N)_2$ and $(R)_2$ for $L_2 (u)=0$ in $\Omega_0$.
Thus it suffices to prove that the $L^2$ Neumann and regularity problems
for $L_2(u)=0$ in $\Omega_0$ are solvable and 
the bounding constants in the nontangential maximal function estimates
depend only on $d$ and $\mu$.
We shall give the proof for the solvability of $(N)_2$.
The proof  for $(R)_2$ is similar.

Thus, let $u\in W^{1,2}(\Omega_0)$ be a weak solution to
the Neumann problem for $L_2 (u)=0$ in $\Omega_0$ with data $g
\in C(\partial\Omega_0)$.
We will show that 
\begin{equation}\label{special-case-5}
\int_{B(P, 1)\cap\partial\Omega_0} |N(\nabla u)|^2\, d\sigma
\le C\int_{\partial\Omega_0} |g|^2\, d\sigma +C \int_{\Omega_0} |\nabla u|^2\, dX
\end{equation}
for any $P\in \partial\Omega_0$.
The desired estimate
$\|N(\nabla u)\|_2 \le C\| g\|_2$ follows from (\ref{special-case-5}) by a simple
argument (see the proof of Theorem \ref{e-Neumann-theorem} in Section 11).

To see (\ref{special-case-5}), we first note that $A_2(x,t)=I$ for $(x,t)\in \Omega_0
\setminus T(x_0,9)$. Hence, 
if $P\in \partial\Omega_0\setminus
\{ (y,0): |y|<11\}$, $u$ is harmonic in $B(P,2)\cap\Omega_0$. As a result,
estimate (\ref{special-case-5}) follows readily from
the solvability of the $L^2$ Neumann problem for $\Delta v=0$ in Lipschitz domains
\cite{Jerison-Kenig-1981-Bull}
and Theorem 6.10
in \cite{KP-1993}.
To show (\ref{special-case-5}) for $P=(y,0)$ with $|y|<11$, 
we use the fact that $A_2(x,t)=A_3(x,t)$ if $t\in [-8,8]$, where
$$
A_3(x,t)=\varphi_1 (x) A(x,0)+ (1-\varphi_1 (x)) I.
$$
Let $L_3=-\text{div}(A_3\nabla)$. Then $L_3(u)=0$ in $B(x_0,20) \times (0,8)$.
Since $A_3(X)$ is independent of the variable $t$, 
it follows from Theorem 8.1 in \cite{KP-1993} (also see Theorem 1.12 in \cite{Hofmann-preprint})
that both $(N)_2$ and $(R)_2$ for $L_3(v)=0$ in $\mathbb{R}^{d+1}_+$
are solvable.
Consequently, estimate (\ref{special-case-5}) 
for $P=(y,0)$ with $|y|\le 11$
follows from Theorem \ref{Neumann-localization}.
This concludes the proof of Theorem \ref{special-case-theorem}.
\end{proof}

\section{Uniform estimates in bounded Lipschitz domains}

In this section we give the proof of Theorems 
\ref{e-Dirichlet-theorem}, \ref{e-Neumann-theorem} and
\ref{e-regularity-theorem} stated in the Introduction.
Let $\Omega$ be a bounded Lipschitz domain in $\mathbb{R}^{d+1}$, $d\ge 2$.
Note that under the conditions
(\ref{real-symmetric}), (\ref{ellipticity}) and (\ref{smoothness-condition}),
the $L^p$ Dirichlet problem for $\mathcal{L}_\varepsilon (u_\varepsilon)=0$
is solvable in $\Omega$ for $2-\delta<p<\infty$, while the $L^p$ Neumann and regularity
problems for $\mathcal{L}_\varepsilon (u_\varepsilon)=0$ 
in $\Omega$ are solvable for $1<p<2+\delta$.
This follows directly from \cite{Dahlberg-1986} and \cite{KP-1993}.
However, without the periodicity condition, the constants $C$ (and $\delta$) 
in the nontangential maximal function estimates in general depend on the 
parameter $\varepsilon$.

The proof of Theorems \ref{e-Dirichlet-theorem}, \ref{e-Neumann-theorem}
and \ref{e-regularity-theorem} relies on the following two observations:

1. The class of elliptic operators $\{ \mathcal{L}_1=
-\text{div}(A\nabla)\}$ with $A(X)$ satisfying
the conditions in Theorem \ref{e-Dirichlet-theorem} is invariant
under rotation of the coordinate system.

2. Suppose $\mathcal{L}_\varepsilon (u_\varepsilon)=0$ in $D=\{(x,t): x\in \mathbb{R}^d
\text{ and } t>\psi(x)\}$. 
Let $w(X)=u_\varepsilon (\varepsilon X)$. Then $\mathcal{L}_1 (w)=0$ in
$$
D_\varepsilon=\big\{ (x,t)\in \mathbb{R}^{d+1}:\ x\in \mathbb{R}^d
\text{ and } t>\psi_\varepsilon (x)\big\},
$$
where $\psi_\varepsilon (x)=\varepsilon^{-1} \psi (\varepsilon x)$.

Since $\|\nabla \psi_\varepsilon\|_\infty =\|\nabla \psi\|_\infty$, it follows from
the observation (2) and Theorem \ref{Dirichlet-theorem-Lipschitz} that
solutions of the $L^2$ Dirichlet problem for $\mathcal{L}_\varepsilon (u_\varepsilon)
=0$ in $D$ satisfies the estimate $\|(u_\varepsilon)^*\|_{L^2(\partial D)}
\le C\|u_\varepsilon\|_{L^2(\partial D)}$ uniformly in $\varepsilon>0$.
Similarly, by Theorems \ref{Neumann-theorem-Lipschitz}
and \ref{regularity-theorem-Lipschitz},
 solutions of the $L^2$ Neumann and regularity problems for
$\mathcal{L}_\varepsilon (u_\varepsilon)=0$ in $D$ satisfy
the uniform estimates
$$
\| N(\nabla u_\varepsilon)\|_{L^2(\partial D)}
\le C \|{\partial u_\varepsilon}/\partial \nu_\varepsilon\|_{L^2(\partial D)}
\quad
\text{ and } \quad
 \| N(\nabla u_\varepsilon)\|_{L^2(\partial D)}
\le C\|\nabla_{tan} u_\varepsilon\|_{L^2(\partial D)},
$$
respectively.

We begin with the Dirichlet problem for $\mathcal{L}_\varepsilon (u_\varepsilon)=0$ in $\Omega$.

\medskip

\noindent{\bf Proof of Theorem \ref{e-Dirichlet-theorem}.}
Although Theorem \ref{e-regularity-theorem} implies Theorem \ref{e-Dirichlet-theorem},
we give a direct proof.

Let $\omega_\varepsilon^X$ denote the $\mathcal{L}_\varepsilon$-harmonic measure
on $\partial\Omega$ and $K_\varepsilon (X,P)=d\omega_\varepsilon^X/d\sigma$,
where $X\in \Omega$ and $P\in \partial\Omega$.
Let $I(P,r)=B(P,r)\cap \partial\Omega$.
We will show that for any $P_0\in \partial\Omega$ and $0<r<r_0$,
\begin{equation}\label{reverse-Holder-e}
\left\{\frac{1}{r^d}
\int_{I(P_0,r)}
|K_\varepsilon (X,P)|^2\, d\sigma(P)\right\}^{1/2}
\le \frac{C}{r^d}
\int_{I(P_0,r)} K_\varepsilon (X, P)\, d\sigma(P),
\end{equation}
where $X\in \Omega\setminus B(P_0,4r)$ and
$C$ is independent of $\varepsilon$.

Fix $P_0\in \partial\Omega$ and $0<r<r_0$. 
Since the class of elliptic operators $\{-\text{div}(A\nabla)\}$
with $A(X)$ satisfying the conditions of Theorem \ref{e-Dirichlet-theorem}
is invariant under rotation,
 we may assume that
\begin{equation}\label{local-graph-representation}
\aligned
\Omega\cap B(P_0,Cr_0)
&=\big\{ (x,t): x\in \mathbb{R}^d \text{ and } t>\psi (x)\big\}
\cap B(P_0, Cr_0)\\
& =D\cap B(P_0, Cr_0)
\endaligned
\end{equation}
and $A(x,t+t_0)=A(x,t)$ for some $t_0\ge 1$ and any $(x,t)\in \mathbb{R}^{d+1}$.
Without loss of generality,
we may also assume that $t_0=1$.
Let $G_\Omega (X,Y)$ and $G_D (X,Y)$ denote the Green's functions for
$\mathcal{L}_\varepsilon$ on $\Omega$ and $D$ respectively.
Fix $Y\in D\setminus B(P_0,4r)$. By the comparison principle (\ref{comparison-principle}), 
we have
$$
\frac{G_\Omega (X, A_s(P))}{G_D(Y, A_s(P))}
\approx 
\frac{G_\Omega (X, A_r(P_0))}{G_D(Y, A_r(P_0))} \quad \text{ for }\
0<s<r \text{ and } P\in I(P_0,r),
$$
where $A_s(P)=P+(0, \dots, 0, c_0 s)$. It follows that
$$
\aligned
\frac{\omega_\Omega^X (I(P,s))}{s^d}
&\approx
\frac{G_\Omega (X, A_s(P))}{s}
\approx 
\frac{G_D (Y, A_s(P))}{s}
\cdot \frac{G_\Omega (X, A_r(P_0))}{G_D (Y, A_r(P_0))}\\
&\approx
\frac{\omega_D^Y (I(P,s))}{s^d}
\cdot
\frac{G_\Omega (X, A_r(P_0))}{G_D (Y, A_r(P_0))},
\endaligned
$$
where $\omega_D$ denotes the $\mathcal{L}_\varepsilon$-harmonic measure on $\partial D$.
This implies that for any $P\in I(P_0,r)$,
$$
K_\varepsilon (X, P)
\le C K_{\varepsilon, D} (Y, P)
\cdot
\frac{G_\Omega (X, A_r(P_0))}{G_D (Y, A_r(P_0))},
$$
where $K_{\varepsilon, D}(Y,P)=d\omega_\varepsilon^Y/d\sigma$.
Note that by Theorem \ref{Dirichlet-theorem-Lipschitz}
and a simple rescaling argument, 
$K_{\varepsilon, D}(Y,P)$ satisfies the reverse
H\"older inequality with exponent $2$ on $I(P_0,r)$.
Thus,
$$
\aligned
\int_{I(P_0,r)} |K_\varepsilon (X,P)|^2\, d\sigma (P)
&\le C\int_{I(P_0,r)}
|K_{\varepsilon, D}(Y, P)|^2\, d\sigma(P)
\left( 
\frac{G_\Omega (X, A_r(P_0))}{G_D (Y, A_r(P_0))}\right)^2\\
&
\le Cr^{-d}
\big[ \omega_D^Y (I(P_0,r))\big]^2 
\left( 
\frac{G_\Omega (X, A_r(P_0))}{G_D (Y, A_r(P_0))}\right)^2\\
&\le
Cr^{-d} 
\big[ \omega_\Omega^X (I(P_0,r))\big]^2.
\endaligned
$$
This gives the desired reverse H\"older inequality (\ref{reverse-Holder-e}).
\qed
\medskip

Next we consider the $L^p$ Neumann problem for $\mathcal{L}_\varepsilon (u_\varepsilon)
=0$ in $\Omega$.

\medskip

\noindent{\bf Proof of Theorem \ref{e-Neumann-theorem}.}
By \cite{KP-1993} it suffices to establish the uniform estimate
$
\|N(\nabla u_\varepsilon)\|_2 \le C\| g\|_2,
$
 where $u_\varepsilon$
is the solution to the $L^2$ Neumann problem in $\Omega$
with boundary data $g\in L^2(\partial\Omega)$
and $\int_{\partial\Omega} g d\sigma =0$.
By dilation we may assume that $|\partial\Omega|=1$.

Fix $P_0\in \partial\Omega$.
As in the proof of Theorem \ref{e-Dirichlet-theorem}, we may assume that
$\Omega\cap B(P_0,Cr_0)$ is given by (\ref{local-graph-representation})
and $A(x,t+1)=A(x,t)$.
By Theorem \ref{Neumann-theorem-Lipschitz}
 and rescaling, the $L^2$ Neumann problem for $\mathcal{L}_\varepsilon (u)=0$
in $D$ is solvable with uniform estimates.
It follows from Theorem \ref{Neumann-localization} that
$$
\int_{I(P_0,r_0)}
|N(\nabla u_\varepsilon)|^2\, d\sigma
\le C\int_{\partial\Omega} |g|^2\, d\sigma
+C\int_{\Omega}
|\nabla u_\varepsilon |^2\, dX.
$$
By covering $\partial\Omega$ with balls of radius $r_0$, we obtain
\begin{equation}\label{final-estimate}
\int_{\partial\Omega} 
|N(\nabla u_\varepsilon)|^2\, d\sigma
\le C\int_{\partial\Omega} |g|^2\, d\sigma
+C\int_\Omega |\nabla u_\varepsilon|^2\, dX.
\end{equation}
To estimate $\int_\Omega |\nabla u_\varepsilon|^2 dX$, we
let $\beta$ be the average of $u_\varepsilon$ on
$\Omega$. Note that
$$
\int_{\partial\Omega} |u_\varepsilon-\beta|^2\, d\sigma
\le C\left\{ \int_\Omega |u_\varepsilon -\beta|^2\, dX
+\int_\Omega |\nabla u_\varepsilon|^2\, dX\right\}
\le C\int_\Omega
|\nabla u_\varepsilon|^2\, dX,
$$
where we have used the Poincar\'e inequality. Thus
$$
\mu \int_{\Omega} |\nabla u_\varepsilon|^2\, d\sigma
\le \int_{\partial\Omega}
g (u_\varepsilon-\beta) \, d\sigma
\le C(\rho) \int_{\partial\Omega} |g|^2\, d\sigma +\rho \int_\Omega |\nabla u_\varepsilon|^2\, dX,
$$
where $0<\rho<\mu$. This implies that $\|\nabla u_\varepsilon\|_{L^2(\Omega)}
\le C\| g\|_2$. In view of (\ref{final-estimate}) we obtain
 $\|N(\nabla u_\varepsilon)\|_2 \le C\| g\|_2$.
\qed
\medskip

Finally we give the proof of Theorem \ref{e-regularity-theorem}.

\medskip

\noindent{\bf Proof of Theorem \ref{e-regularity-theorem}.}
By \cite{KP-1993} it suffices to show that
$$
\|N(\nabla u_\varepsilon)\|_2 \le C\|\nabla_{tan} f\|_2
+|\partial\Omega|^{\frac{1}{1-d}} \| f\|_2,
$$ 
where
$u_\varepsilon$ is the solution of the $L^2$ regularity
problem for $\mathcal{L}_\varepsilon (u_\varepsilon)
=0$ in $\Omega$ with data $f\in W^{1,2}(\partial\Omega)$.
By dilation we may assume that $|\partial\Omega|=1$.

By a localization argument similar to that in the proof of Theorem
\ref{e-Neumann-theorem}, we obtain
$$
\int_{\partial\Omega}
|N(\nabla u_\varepsilon)|^2\, d\sigma
\le C\int_{\partial\Omega} |\nabla_{tan} f|^2\, d\sigma
+C\int_\Omega |\nabla u_\varepsilon|^2\, dX.
$$
Let
$g=\frac{\partial u_\varepsilon}{\partial \nu_\varepsilon}$.
The desired estimate follows from the observation
$$
\mu \int_{\Omega}
|\nabla u_\varepsilon|^2\, dX
\le \rho \| g\|^2_2+
C(\rho) \| u_\varepsilon\|^2_2 
\le C\rho \| N(\nabla u_\varepsilon)\|^2_2
+ C (\rho) \|f\|^2_2,
$$
where we have used $\| g\|_2 \le C\| N(\nabla u_\varepsilon)\|_2$.
\qed
\medskip

\begin{remark}
{\rm
Suppose $A(X)$ is periodic and $C^\alpha (\mathbb{R}^{d+1})$ for some $\alpha>0$,
Then it satisfies the condition (\ref{smoothness-condition}).
Furthermore, by the $L^\infty$ gradient estimate in \cite{AL-1987},  
$$
|\nabla u_\varepsilon (X)|\le C\left(\frac{1}{|B(X,r)|}
\int_{B(X,r)} |\nabla u_\varepsilon(Y)|^2\, dY\right)^{1/2}
$$
if $\mathcal{L}_\varepsilon (u_\varepsilon) =0$ in $B(X,2r)$.
Consequently, the solutions of the $L^p$ Neumann and regularity problems
in Theorems \ref{e-Neumann-theorem} and \ref{e-regularity-theorem}
satisfy the uniform estimates
$$
\| (\nabla u_\varepsilon)^*\|_p \le C\| \partial u_\varepsilon/\partial\nu_\varepsilon\|_p
\quad
\text{ and }
\quad
\| (\nabla u_\varepsilon)^*\|_p
\le C\|\nabla_{tan} u_\varepsilon\|_p
$$
for $1<p<2+\delta$.
}
\end{remark}

\section{Appendix: Dahlberg's Theorem}

In this appendix we give Dahlberg's original proof of Lemma \ref{key-Dirichlet-lemma}
(and thus Theorem \ref{Dahlberg-theorem}).

\begin{lemma}\label{difference-lemma}
Let $\mathcal{L}=-\text{div}(A\nabla)$
with coefficients satisfying (\ref{real-symmetric}),
(\ref{ellipticity}) and (\ref{periodic-in-t}).
Let $u$ be a weak solution of $\mathcal{L}(u) =0$ in 
$\Omega=B(x_0,r)\times (t_0-r, t_0+r)$ for some $(x_0,t_0)\in \mathbb{R}^{d+1}$ and
$r>3$.
Then
\begin{equation} \label{difference-estimate}
|u(x_0, t_0+1)-u(x_0, t_0)|
\le \frac{C}{r}
\left( \frac{1}{|\Omega|}
\iint_\Omega |u(y,s)|^2\, dyds\right)^{1/2},
\end{equation}
where $C$ depends only on $d$ and $\mu$.
\end{lemma}

\begin{proof}
Let $F(x,t)=Q(u)(x,t)$. By the periodicity 
assumption (\ref{periodic-in-t}), $\mathcal{L}(F)=0$
in $B(x_0,r)\times (t_0-r, t_0+r-1)$.
It follows that
\begin{equation}\label{Moser-estimate}
|F(x_0,t_0)|\le 
C\left( \frac{1}{r^{d+1}}
\int_{t_0-\frac{r}{4}}^{t_0+\frac{r}{4}}
\int_{B(x_0,\frac{r}{4})}
|F(y,s)|^2\, dyds\right)^{1/2}.
\end{equation}
The desired estimate then follows from 
$$
|F(y,s)|^2 \le \int_s^{s+1} |\frac{\partial u}{\partial t} (y, t)|^2\, dt
$$
and Cacciopoli's inequality.
\end{proof}

\begin{lemma}\label{boundary-difference-lemma}
If in addition to the assumption in Lemma \ref{difference-lemma}, we also assume that
$u=0$ in $\partial B(x_0,r)\times (t_0-r, t_0+r)$, then
\begin{equation}\label{boundary-difference-estimate}
|u(x,t_0+1)-u(x,t_0)|
\le \frac{C}{r}
\left( \frac{1}{|\Omega|}
\iint_\Omega |u(y,s)|^2\, dyds\right)^{1/2}
\end{equation}
for any $x\in B(x_0,r)$,
where $C$ depends only on $d$ and $\mu$.
\end{lemma}

\begin{proof}
The proof is similar to that of Lemma \ref{difference-lemma}. In the place of (\ref{Moser-estimate}),
one uses the boundary $L^\infty$ estimate.
\end{proof}

\noindent{\bf Proof of Lemma \ref{key-Dirichlet-lemma}.}

The case $1< r\le 10$ follows easily
by covering $B(x_0,r)$ with balls of radius $1$.
If $r>10$, we also
cover $B(x_0,r)$ with a sequence of balls 
$\{ B(x_k,1)\}$ of radius $1$ such that
$$
B(x_0,r) \subset \bigcup_k B(x_k, 1)\subset B(x_0,r+1)
\quad \text{ and }\quad
\sum_k \chi_{B(x_k,1)}\le C.
$$
Let $I$ denote the left-hand side of (\ref{local-Dirichlet-condition}).
It follows from the assumption (\ref{local-Dirichlet-condition}) (with $r=1$)
and boundary Harnack inequality (\ref{local-bound}) that
\begin{equation}\label{a-1}
I\le C\, \sum_k \int_{T(x_k,2)} |u(X)|^2\, dX \le C\, \sum_k |u(x_k, 1)|^2.
\end{equation}

To estimate the right-hand side of (\ref{a-1}), we let
$v_1$ be the Green's function for the operator $\mathcal{L}$ on the domain
$\Omega_1= B(x_0, 2r)\times (0, r)$ with pole at
$(x_0, r/2)$. Similarly, we let $v_2$ be the Green's function for $\mathcal{L}$ on $\Omega_2=
B(x_0,2r)\times (0, 2r)$ with pole at $(x_0, 3r/2)$.
Note that
$$
v_i (x_0, r/4)\ge c r^{2-d} \quad \text{ for } i=1,2.
$$
Thus, by the comparison principle (\ref{comparison-principle}), we have
$$
\frac{u(x_k,1)}{v_i(x_k,1)}
\le C\, \frac{u(x_0,r/4)}{v_i (x_0,r/4)}
\le Cr^{d-2} u(x_0, r/4)
$$
for $i=1, 2$. 
In view of (\ref{a-1}), this leads to
\begin{equation}\label{Dirichlet-estimate-4}
I\le C r^{2(d-2)} |u(x_0, r/4)|^2 \sum_k v_1 (x_k, 1) v_2 (x_k, 1).
\end{equation}

Next we let $\omega$ denote the $\mathcal{L}$-harmonic measure on the domain $\Omega_1$,
evaluated at the point $(x_0, r/2)$.
Since 
\begin{equation}\label{Dirichlet-estimate-5}
\aligned
v_1 (x_k, 1) &\le C\, \int_{B(x_k,1)\times\{ 0\}} d\omega,\\
v_2 (x_k,1) &\le C\, v_2 (y,1) \quad \text{ for any } y\in B(x_k, 1),
\endaligned
\end{equation}
the right-hand side of (\ref{Dirichlet-estimate-4}) is bounded by
$$
Cr^{2(d-2)} |u(x_0, r/4)|^2 \sum_k \int_{B(x_k,1)\times \{ 0\} } v_2(y,1) d\omega .
$$
It follows that
\begin{equation}\label{Dirichlet-estimate-6}
I\le C\, r^{2(d-2)} |u(x_0, r/4)|^2 \int_{B(x_0, 2r)\times\{ 0\} } v_2(y,1) d\omega .
\end{equation}

We now consider the function
\begin{equation}\label{Dirichlet-estimate-7}
F(x,t)=v_2 (x,t+1)-v_2 (x,t)
\quad \text{ in } \Omega_1.
\end{equation}
By the periodicity assumption (\ref{periodic-in-t}), $\mathcal{L}(F)=0$ in $\Omega_1$. It follows that
\begin{equation}\label{Dirichlet-estimate-8}
\aligned
 F(x_0, r/2) & =\int_{\partial\Omega_1} F d\omega\\
&=\int_{B(x_0,2r)\times \{ 0\}} v_2 (y,1)\, d\omega 
+\int_{B(x_0, 2r)\times\{ r\}} F(y, r)\, d\omega ,
\endaligned
\end{equation}
where we have used the fact that $F=0$ on $\partial B(x_0,2r)\times (0,r)$
and $F(y,0)=v_2(y,1)$ on $B(x_0, 2r)\times \{ 0\}$.
We claim that
\begin{eqnarray}
|F(x_0, r/2)| & \le & C\, r^{1-d} \label{claim-1},\\
|F(y,r)| & \le & C r^{1-d} \quad \text{ if } |y-x_0|< 2r.
\label{claim-2}
\end{eqnarray}
Assume the claim (\ref{claim-1})-(\ref{claim-2}) for
a moment. It then follows from (\ref{Dirichlet-estimate-8}) that
$$
\int_{B(x_0,2r)\times \{ 0\} }
v_2 (y,1)\, d\omega \le C\, r^{1-d}.
$$
In view of (\ref{Dirichlet-estimate-6}), we obtain
\begin{equation}\label{Dirichelt-estimate-11}
I\le C\, r^{d-3} |u(x_0, r/4)|^2 \le \frac{C}{r^3}
\int_{T(x_0,r)} |u(X)|^2\, dX.
\end{equation}

Finally we note that since $v_2(x,t)\le Cr^{2-d}$ for $(x,t)\in B(x_0, 2r)\times (0,r+3)$,
estimates (\ref{claim-1}) and (\ref{claim-2}) follow readily from
Lemmas \ref{difference-lemma} and \ref{boundary-difference-lemma}, respectively.
This completes the proof of Lemma \ref{key-Dirichlet-lemma}.
\qed

\bibliography{ks1}

\small
\noindent\textsc{Department of Mathematics, 
University of Chicago, Chicago, IL 60637}\\
\emph{E-mail address}: \texttt{cek@math.uchicago.edu} \\

\noindent\textsc{Department of Mathematics, 
University of Kentucky, Lexington, KY 40506}\\
\emph{E-mail address}: \texttt{zshen2@email.uky.edu} \\

\noindent \today

\end{document}